\documentclass[a4paper]{article}
\usepackage[utf8]{inputenc}
\usepackage[english]{babel}
\usepackage{amsmath}
\usepackage{amssymb}
\usepackage{amsthm}
\usepackage{graphicx,epstopdf}
\usepackage[ruled,linesnumbered]{algorithm2e}
\usepackage{tikz,tikz-3dplot}
\usetikzlibrary{calc}
\usetikzlibrary{decorations.pathmorphing}
\usepackage{pgfplots}
\usepackage{svg}
\usepackage{xcolor}
\usepackage{float}
\usepackage{booktabs,multirow}
\usepackage{amstext} 
\usepackage{array} 
\usepackage{hyperref}
\usepackage{enumerate}
\usepackage{titling}
\usepackage{a4wide}

\usepackage{csquotes}

\newtheorem{theorem}{Theorem}

\newtheorem{proposition}[theorem]{Proposition}
\newtheorem{corollary}[theorem]{Corollary}
\theoremstyle{definition}
\newtheorem{example}[theorem]{Example}

\newcommand{\dir}{_{\rm dir}}
\newcommand{\poi}{_{\rm poi}}
\newcommand{\R}{\mathbb{R}}
\newcommand{\transpose}{{\raisebox{\depth}{\footnotesize $\intercal$}}} 

\DeclareMathOperator{\conv}{conv}
\DeclareMathOperator{\cone}{cone}
\DeclareMathOperator{\recc}{0^+\!}
\DeclareMathOperator{\vertex}{vert}
\DeclareMathOperator*{\argmin}{argmin}
\DeclareMathOperator{\Min}{Min}

\DeclareMathOperator{\inter}{int}

\DeclareMathOperator{\dom}{dom}
\DeclareMathOperator{\epi}{epi}
\DeclareMathOperator{\bd}{bd}

\newcommand*{\colonequals}{\mathrel{\vcenter{\baselineskip0.5ex%
      \lineskiplimit0pt\hbox{\scriptsize.}\hbox{\scriptsize.}}}=}
\renewcommand{\P}{\mathcal{P}}
\renewcommand{\O}{\mathcal{O}}
\newcommand{\extR}{\bar{\R}}
\providecommand{\bt}{{bensolve tools}}
\providecommand{\baron}{BARON}
\providecommand{\octave}{{\sc Gnu Octave}}

\begin{document}
\title{A vector linear programming approach for certain global optimization problems\textsuperscript{1}}

\author{Daniel Ciripoi\textsuperscript{2}
  \and
  Andreas Löhne\textsuperscript{2}
  \and
  Benjamin Weißing\textsuperscript{2}
}

\footnotetext[1]{This research was supported by the German Research Foundation (DFG) grant number LO--1379/7--1.}
\footnotetext[2]{Friedrich Schiller University Jena, Department of
    Mathematics, 07737 Jena, Germany,
    [daniel.ciripoi\textbar andreas.loehne\textbar benjamin.weissing]@uni-jena.de
}
\maketitle

\begin{abstract}
Global optimization problems with a quasi-concave objective function and linear constraints are studied.
We point out that various other classes of global optimization problems can be expressed in this way.
We present two algorithms, which can be seen as slight modifications of Benson-type
algorithms for multiple objective linear programs (MOLP).
The modification of the MOLP algorithms results in a more efficient treatment
 of the studied optimization problems.
This paper generalizes results of Schulz and Mittal \cite{mittal_schulz} on quasi-concave problems and
Shao and Ehrgott \cite{linear_multiplicative} on multiplicative linear programs. Furthermore,
it improves results of L\"ohne and Wagner \cite{dc_paper} 
on minimizing the difference $f=g-h$ of two convex functions $g$, $h$ where
either $g$ or $h$ is polyhedral.
Numerical examples are given and the results are compared with
the global optimization software \baron\  \cite{baron}, \cite{baron_2}.

\medskip
\noindent
{\bf Keywords:} global optimization, DC programming, multiobjective linear programming, linear vector optimization
\medskip

\noindent
{\bf MSC 2010 Classification:} 90C26, 90C29, 52B55
\end{abstract}

\section{Introduction}

The object of study is a global optimization problem with a
quasi-concave objective function
$f \colon \mathbb{R}^q \to \extR$ and
linear constraints of the form
\begin{equation*}\label{p}
	\min f(Px) \quad  \text{s.t.} \quad Ax \geq b\text{,} \tag{QCP}
\end{equation*}
where $P \in \mathbb{R}^{q \times n}$, $A \in \mathbb{R}^{m \times n}$
and $b \in \mathbb{R}^m$.  The symbol
\(\extR\colonequals \R\cup\{\,\pm\infty\,\}\) denotes the set of
extended reals.  The following four examples show classes of global
optimization problems which are covered by \eqref{p}.

\begin{example}[DC programming - ``convex component'' being polyhedral]\label{ex01} 
	Consider
	\begin{equation} \label{dc1}
		\min_{x \in \dom g} [g(x) - h(x)]
	\end{equation}
	where $g: \mathbb{R}^n \to \mathbb{R}\cup \{+\infty\}$ is polyhedral convex and $h: \mathbb{R}^n \to \mathbb{R}\cup \{+\infty\}$ is convex.
	The reformulation
	\begin{equation*}
		\min_{x,\, r} \left[r - h(x)\right] \quad\text{s.t.}\quad (x,r) \in \epi g
	\end{equation*}
	has a concave objective function.  As \(g\) is polyhedral, the
        constraints are linear.
        See, e.g.\ \cite{dc_paper} for more details.
\end{example} 

\begin{example}[DC programming - ``concave component'' being polyhedral]\label{ex02} 
  Consider the DC program \eqref{dc1} but, in contrast to Example
  \ref{ex01}, let $g$ be convex and $h$ be polyhedral convex and proper.  As
  shown in \cite{dc_paper}, the Toland-Singer dual problem
  \begin{equation}\label{eq_tsd}
    \min_{y \in \dom h^*} [h^*(y) - g^*(y)]
  \end{equation}
  can be utilized to solve \eqref{dc1}, where
  $g^*: \R^n \to \R\cup \{+\infty\}$ and
  $h^*: \R^n \to \R\cup \{+\infty\}$ are the conjugates of $g$ and
  $h$, respectively.  Since $h^*$ is polyhedral convex and $g^*$ is
  convex, we can proceed as in Example \ref{ex01} to obtain a
  reformulation as \eqref{p}.
\end{example} 

\begin{example}[Minimizing a convex function over the boundary of a polytope]\label{ex03}
	Let $g: \mathbb{R}^q \to \mathbb{R}\cup \{+\infty\}$ be convex and let
	\begin{equation*}
		Q = \left\{ x \in \mathbb{R}^q \,\middle\vert\,\exists u \in \R^k\colon A x + B u \geq b\right\}
	\end{equation*}
        be a polytope.  We consider the problem
	\begin{equation} \label{p_ex03_a}
		\min g(x) \quad \text{s.t.} \quad x \in \bd Q\text{,}
	\end{equation}
	where $\bd Q$ denotes the boundary of $Q$.\par
	We assume $0 \in \inter Q$ and that $g$ is Lipschitz over $Q$.
	For some real parameter $c > 0$ we define $h_c: \R^q \to \R\cup\{+\infty\}$ by its epigraph:
	\begin{equation*}
		\epi h_c = \left\{(x,r)  \in \R^q \times \R\,\middle\vert\,
		\exists u \in \R^k\colon
		A x + B u - \frac{1}{c}\, b\, r \geq 0,\; r \geq 0\right\}\text{.}
	\end{equation*}
        Because $\epi h_c$ is equal to the cone generated by the set $Q
        \times \{c\}$,  the negative of the polyhedral convex function $h(x) = h_c(x)
        - c$ penalizes points belonging to the interior of $Q$.  Thus, for $c$
        chosen sufficiently large, \eqref{p_ex03_a} can be replaced by
        the equivalent problem
	\begin{equation} \label{p_ex03_b}
		\min g(x) - (h_c(x) - c) \quad \text{s.t.} \quad x \in Q\text{.}
	\end{equation}
	Problem \eqref{p_ex03_b} is a DC program as considered and transformed into \eqref{p} in Example \ref{ex02}. Note that $g$ needs to be modified by setting $g(x) = \infty$ for $x \not \in Q$. For more details see Section \ref{subsec:bd} below.
\end{example} 

\begin{example}[Linear multiplicative programming \cite{linear_multiplicative}]\label{ex04} 
  A special instance of \eqref{p} is to minimize the product of affine
  functions under linear constraints:
  \begin{equation}\label{eq:mmp}
    \min \prod_{i=1}^q (c_i^{\transpose} x + d_i) \quad \text{ s.t. } \quad A x \geq b\text{.}
  \end{equation}
  Here, we assume $c_i^\intercal x + d_i > 0$ for feasible $x$.
  Various applications of this problem class can be found in the
  literature, see e.g.\ \cite{linear_multiplicative}.
\end{example}\par

Our approach to solve \eqref{p} can be summarized as follows: We show
that solving \eqref{p} is equivalent to solve
\[
  \min_{y\in\vertex\P} f(y)\text{,}
\]
where \(\vertex\cdot\) denotes the vertex set of a polyhedron, and
\begin{equation*}
  \P = \left\{\, y \in \R^q \middle |\,   y - P x \in C,\; Ax \geq b\,\right\}
\end{equation*}
denotes the {\em upper image} of the {\em
  vector linear program}
\begin{equation}\label{eq:vlp}
  \min\nolimits_C Px \quad \text{ s.t. } \quad Ax \geq b \text{.} \tag{VLP}
\end{equation}
Here, $C$ is some polyhedral convex pointed cone with respect to which $f$ is
monotone, that is $y-x \in C$ implies $f(x) \leq f(y)$.  A vector
linear program describes the minimization of the linear function $Px$
under the constraints $Ax \geq b$ with respect to the partial ordering
induced by $C$.  For further information and applications compare, for
example, \cite{benson_type}.  Algorithms designed for solving vector
linear programs, in particular Benson-type algorithms
\cite{benson_type}, also compute the set $\vertex\P$.
Thus, those algorithms could be utilized directly to solve \eqref{p}.
However, computing all the vertices of $\mathcal P$ may be too
expensive in practice.  Therefore we alter the algorithms for solving
\eqref{eq:vlp} slightly by introducing certain bounding techniques,
which are introduced in \cite{linear_multiplicative} for the special
case of linear multiplicative programming \eqref{eq:mmp}.  These
bounding techniques usually lead to a decline in the number of
vertices of $\P$ that need to be computed.

An overview over various solution techniques for global optimization
problems can be found in \cite{global_optimization,handbook}.
Quasi-concave minimization problems have been investigated, for
instance, in \cite{rank_two,majthay}.  The idea to solve a (scalar)
global optimization problem via a multiple objective linear program
(MOLP) (which we understand to be a vector linear program with the
special cone $C=\R^q_+$) is not new in the literature.  F{\"u}l{\"o}p
\cite{Fulop93} shows that a linear bilevel programming problem can be
solved by optimizing a linear function over the Pareto set of a MOLP.
Mittal and Schulz \cite{mittal_schulz} minimize a quasi-concave
objective function under linear constraints via a corresponding MOLP.
Shao and Ehrgott \cite{linear_multiplicative} investigate the special
case of multiplicative linear programs using this idea. L\"ohne and
Wagner \cite{dc_paper} solve DC optimization problems with one
polyhedral component (compare Examples \ref{ex01} and \ref{ex02}) by
utilizing a MOLP solver.

The article is organized as follows. In Section~\ref{sec_prel} we introduce some concepts and notation.
The problem formulation and corresponding concepts and results are given in Section~\ref{sec_approach}.
Section~\ref{sec_prim} is devoted to a first algorithm, which we call
the {\em primal algorithm} as it is a modification of (the primal
version of) Benson's algorithm \cite{benson, benson_type} for vector linear programs.
Section~\ref{sec_dual} deals with the {\em dual algorithm}, which is a modification of the dual variant of Benson's algorithm for VLP \cite{dual_variant, benson_type}.
We also recall some facts about {\em geometric duality} for vector linear programs \cite{geom_duality}.
In Section~\ref{nonsolid_cones}, our methods are extended to the case
of non-solid cones, which requires a problem reformulation in order to
be able to use these methods for VLPs.
The last section provides numerical examples.

\section{Preliminaries} \label{sec_prel}

A {\em polyhedral convex set} or {\em convex polyhedron} is defined to
be the solution set of a system of finitely many affine inequalities.
Since all polyhedral sets in this article are convex, we will say {\em
  polyhedral set} or {\em polyhedron} for short.  If a polyhedron is given as in the latter definition, we speak about an {\em H-representation} of the polyhedron.
The well-known Minkowski-Weyl theorem states that every nonempty
polyhedron $K \subseteq \R^q$ can be represented as a {\em generalized
  convex hull} of finitely many points $\{\,v^1,\dots,v^r\,\} \subseteq \R^q$, $r \geq 1$ and finitely many directions $\{\,d^1,\dots,d^s\,\}\subseteq \R^q$, $s \geq 0$, that is,
\begin{equation*}
	K = \left\{\sum_{i=1}^r \lambda_i v^i + \sum_{j=1}^s \mu_j d^j \,\middle\vert\,\lambda_i \geq 0\; (i=1,\dots,r),\; \mu_j \geq 0\; (j=1,\dots,s),\; \sum_{i=1}^r \lambda_i = 1\right\}\text{.}
\end{equation*}
The pair $(K\poi,K\dir)$ consisting of the two sets $K\poi \colonequals  \{\,v^1,\dots,v^r\,\}$ and $K\dir \colonequals  \{\,d^1,\dots,d^r\,\}$ is called {\em V-representation} of $K$.
We also write
\begin{equation*}
K=\conv K\poi + \cone K\dir\text{,}
\end{equation*}
where $\conv\cdot$ denotes the convex, and $\cone\cdot$ the conical
hull of a set.  We assume that $K\poi$ is nonempty and define
$\cone \emptyset \colonequals \{0\}$ as $K\dir$ is allowed to be
empty.\par
A polyhedron $K$ can be expressed as
\begin{equation*}
	K = \left\{x \in \R^q \,\middle\vert\,  \exists u \in
          \R^k\colon A x + B u \geq b\right\}\text{,}
\end{equation*}
where $A \in \R^{m\times q}$, $B \in \R^{m \times k}$ and
$b \in \R^m$.  This type of representation is referred to as {\em
  projection-} or {\em P-representation}, as $K$ is the projection of
the polyhedron
$Q = \left\{(x,u) \in \R^q \times \R^k \,\middle\vert\, A x + B u \geq
  b \right\}$ onto $\R^q$.\par
A {\em multiple objective linear program} (MOLP) is an optimization
problem of the form
\begin{equation}\label{eq:molp}
\min Px \quad \text{ s.t. } \quad Ax \geq b \tag{MOLP}\text{,}
\end{equation}
where $P \in \mathbb{R}^{q\times n}$, $A \in \mathbb{R}^{m\times n}$ and $b \in \mathbb{R}^m$.
Typically we have at least two linear objective functions, i.e.\ 
$q\geq 2$.  The operator $\min\cdot$ in \eqref{eq:molp} is to be
understood with respect to the component-wise partial ordering in
$\mathbb{R}^q$: $y \leq z$ if and only if $z-y \in \mathbb{R}^q_+ \colonequals 
\{ \, w \in \mathbb{R}^q \mid  w_1\geq 0,\dots, w_q \geq 0\,\}$.
If the cone $\mathbb{R}^q_+$ is replaced by a general
polyhedral convex pointed cone $C
\subseteq \mathbb{R}^q$, we obtain a {\em vector linear program}
(VLP):
\begin{equation}\label{vlp}
	\min\nolimits_C Px \quad \text{s.t.} \quad Ax \geq b \tag{VLP}\text{.}
\end{equation}
For a polyhedral convex pointed cone $C \subseteq \mathbb{R}^q$ there
exist matrices $Y \in \mathbb{R}^{q \times o}$ and $Z \in\mathbb{R}^{q
  \times p}$, $o,p \in \mathbb{N}$, such that 
\begin{equation}\label{def_C}
	C=\left\{\, Y \lambda \,\middle\vert\,  \lambda \in \mathbb{R}^o_+
  \,\right\}=\left\{\, y \in \mathbb{R}^q \,\middle\vert\,   Z^\transpose y \geq 0 \,\right\} \text{.}
\end{equation}
The equivalence $x \leq_C y \iff Z^\transpose x \leq Z^\transpose y$
follows.  Elements of
$S\colonequals \left\{\, x \in \mathbb{R}^n \mid Ax \geq b \,\right\}$
are called {\em feasible points}.  Elements of
$0^+ S \colonequals \left\{\, x \in \mathbb{R}^n \mid Ax \geq
  0\,\right\}$, the recession cone of $S$, are {\em feasible
  directions}.  By
\(P[S]\colonequals \left\{Px\,\middle\vert\, x\in S\right\}\) we
denote the image of \(S\) under \(P\).  The polyhedron
$\P \colonequals P[S]+C$ is known as {\em upper image} of \eqref{vlp}.
\par

We call a point $y \in \R^q$ a {\em minimal point} of the polyhedron
$P \subseteq \R^q$ if there is no $z \in P$ with $z \leq_C y$ and
$z \neq y$.  The set of minimal points of $P$ is denoted by
$\Min_C P$.  A vector $x \in S$ is called {\em minimizer} of
\eqref{vlp} if $Px \in \Min_C P[S]$.  A feasible direction
$x \in 0^+S$ is called {\em minimizer} of \eqref{vlp} if
$Px \in \Min_C P[0^+ S]\setminus\{\,0\,\}$.  Let $S\poi \subseteq S$
and $S\dir \subseteq 0^+ S$ with \(P[S\dir]\cap\{\,0\,\}=\emptyset\),
be finite sets.  We call $(S\poi,S\dir)$ a {\em finite infimizer} of
\eqref{vlp} if
\begin{equation*}
	\conv P[S\poi]+\cone P[S\dir]+C = \P\text{.}
\end{equation*}
A finite infimizer consisting of minimizers only is called a {\em solution} of \eqref{vlp}, see \cite{buch_andreas, benson_type}.

Finally we recall two types of scalarizations for \eqref{vlp}.
For a $w \in \mathbb{R}^q$, the {\em  weighted sum scalarization} is
\begin{equation}\label{scl:p1}\tag{P\textsubscript{1}(\(w\))}
  \min w^\intercal Px \quad \text{s.t.}\quad Ax \geq b\text{.}
\end{equation}
The corresponding dual problem is
\begin{equation}\label{scl:d1}\tag{D\textsubscript{1}(\(w\))}
  \max b^\intercal u \quad \text{s.t.}\; \left\{
    \begin{aligned}
      A^\transpose u &= P^{\transpose} w\text{,}\\
      u &\geq 0\text{.}
    \end{aligned}
  \right.
\end{equation}
Another relevant scalarization is the {\em translative scalarization} (or {\em scalarization by a reference variable}) for some $t \in \mathbb{R}^q$:  
\begin{equation}\label{scl:p2}\tag{P\textsubscript{2}(\(t\))} 
  \min z \quad \text{s.t.}\quad
  \left\{
    \begin{aligned} Ax &\geq b\text{,}\\
      Z^\transpose Px &\leq Z^\transpose t+z \cdot Z^{\transpose} c\text{.}
    \end{aligned}
  \right.
\end{equation} 
Note that the second inequality is equivalent to $Px \leq_C t+z \cdot c$. The purpose of this scalarization method is depicted in Proposition \ref{prop_scal2}.
The corresponding dual problem of \eqref{scl:p2} (in a slightly modified form, see \cite{benson_type} for details) is 
\begin{equation}\label{scl:d2}\tag{D\textsubscript{2}(\(t\))}
  \max b^\intercal u - t^\intercal w \quad \text{s.t.} \quad
  \left\{
    \begin{aligned}
      A^{\transpose} u &=P ^{\transpose} w\text{,}\\
      c^\intercal w&=1\text{,}\\
      Y^{\transpose}w &\geq 0\text{,}\\
      u &\geq 0\text{.}
    \end{aligned}
  \right.
\end{equation}

A function $f \colon \R^q \to \extR$ is said to be {\em
  quasi-concave} if its {\em super level sets} $U_r \colonequals \{\, x \in \R^q
\mid  f(x) \geq r\,\}$ are convex for every $r \in \mathbb{R}$.
Equivalently \cite[Section 3.4]{convex_opt}, $f$ is quasi-concave if
and only if
\begin{equation*}
 \forall \lambda \in (0,1),\; \forall x,y \in \R^q\colon\; f(\lambda x +
  (1-\lambda)y) \geq \min \{ \, f(x),f(y)\,\}\text{.}
\end{equation*}\par
Let $C \subseteq\R^q$ be a pointed (i.e.\ $C \cap (-C) =
\{\,0\,\}$) convex cone.  As usual, we write $x \leq_{C} y$ if $y-x \in
C$.  A function $g \colon \R^q \to \extR$ is said to be {\em
  $C$-monotone} on a set \(D\subseteq\R^q\) if for all
$x,y \in D$, $x \leq_{C} y$ implies $g(x) \leq g(y)$.\par
A function $f \colon \R^q \to \extR$  is called {\em polyhedral convex (polyhedral)}
if its { \em epigraph } $\epi f \colonequals 
\left\{ (x,r) \in  \R^n \times \R \; | \; f(x) \leq r  \right\}$ is a polyhedral
convex set. The {\em domain} of $f$ is defined as 
$\dom f \colonequals \left\{ x \in \R^n \; | \; f(x) < + \infty \right\}$.\par
The {\em conjugate} $f^* \colon \R^q \rightarrow \R \cup \left\{ + \infty \right\}$ of $f$ with $\dom f \neq \varnothing$ is defined as
\begin{equation*}
f^*(x^*)=\sup_{x \in \dom f} \left\{ x^\intercal x^* -f(x) \right\} \, .
\end{equation*}

\section{Problem formulation} \label{sec_approach}

The optimization problem we intend to solve is 
\begin{equation}\label{eq:qcp}\tag{QCP}
  \min f(Px) \quad  \text{s.t.} \quad Ax \geq b \text{,}
\end{equation}
where \(f \colon \R^q \to \extR\) is a quasi-concave function,
$P \in \R^{q \times n}$, $A \in \R^{m \times n}$ are matrices and
$b \in \R^m$ is a vector.  In typical applications one has
\(q \ll n\).  A {\em low rank of non-linearity} (see e.g.\
\cite{rank_two}) is indicated by the projection of the $n$-dimensional
feasible polyhedron $S\colonequals \{x \in \R^n | \; Ax \geq b\}$ onto
the ``low''-dimensional polyhedron \(P[S]\subseteq\R^q\).  In other
words, the problem is non-linear with respect to only \(q\) instead of
\(n\) variables.  The problem can be solved if $q$ is not ``too
large'' (say up to $q=20$).  The low rank property can arise from
modeling techniques, e.g.\ by introducing slack variables, or by
auxiliary variables which are inserted in order to transform
polyhedral convex terms (such as finite maximum or absolute value)
into linear constraints, see e.g.\ \cite{dc_paper} for an example from
location analysis.

We assume that $C\subseteq \R^q$ is a polyhedral convex pointed cone such that:
\makeatletter
\tagsleft@true
\makeatother
\begin{flalign}
  \label{assmpt:M}\tag{M}%
  \hphantom{\mathrm{(M)}}\;&\text{The objective function \(f\) is \(C\)-monotone on the set \(P[S]-C\).}&\\[.8ex]
  \label{assmpt:B}\tag{B}%
  &C \supseteq \recc P[S]\text{, i.e.\ \(P[S]\) is {\em \(C\)-bounded}.}
\end{flalign}
\makeatletter
\tagsleft@false
\makeatother

Note that assumption \eqref{assmpt:M} is always satisfied for the cone
$C=\{0\}$.  There are three reasons for a larger (with respect to set
inclusion) cone $C$.  First, if one is able to find a cone with
$\inter C \neq \emptyset$, a direct application of modified VLP
algorithms is possible, while in the case of $\inter C = \emptyset$ a
reformulation of the problem is necessary, which is discussed in
Section \ref{nonsolid_cones}.  Secondly, a larger cone $C$ tends to
reduce the number of iteration steps required.  Third, a larger cone
$C$ can be necessary to satisfy the boundedness assumption
\eqref{assmpt:B}.\par
A further assumption is made in the algorithms following subsequently:
We assume that an H-representation of an initial outer approximation
\(\O\) is available as input.  In order to ensure that \(\O\)
possesses a vertex, we additionally require \(\O\) to be
\(C\)-bounded.  Thus, the assumption reads as follows:
\makeatletter \tagsleft@true \makeatother
\begin{flalign}\label{assmpt:O}\tag{O}%
  \hphantom{\mathrm{(O)}}\;&\text{An H-representation of a polyhedron \(\O\), with \(\O\supseteq\P\) and
    \(\recc\O = C\), is given.}&
\end{flalign}
\makeatletter \tagsleft@true \makeatother
An initial approximation
according to \eqref{assmpt:O} can be computed whenever assumption
\eqref{assmpt:B} holds.  Appropriate techniques for constructing
\(\O\) can be found in, e.g., \cite{benson_type}.

We next show existence of an optimal solution of \eqref{p} and its
attainment in a vertex of $\P$.

\begin{proposition}\label{prop_exist_minimal}
  Let the assumptions \eqref{assmpt:M} and \eqref{assmpt:B} be
  satisfied.  Let \(\O\) denote a polyhedron according to
  \eqref{assmpt:O}.  Then
  \begin{equation*}
    \min_{y \in \vertex\O} f(y) \leq \inf_{x \in S} f(Px)\text{.}
  \end{equation*}
\end{proposition}
\begin{proof}
  As \(C\) is supposed to be pointed, \(\vertex\O\neq\emptyset\)
  follows from \(\recc\O = C\).  Let \(\vertex\O = \{y^1,\ldots,y^k\}\).
  For any \(x\in S\) there exist \(\lambda_j\in \left[0,1\right]\), \(j=1,\ldots,k\); and
  a direction \(c\in C\) such that
  \[
    Px = \sum_{j=1}^k\lambda_jy^j + c\text{.}
  \]
  From \eqref{assmpt:M} follows \(f(Px)\geqslant f(Px - c)\geqslant
  f\left(\sum_{j=1}^k\lambda_j y^j\right)\).  As \(f\) is supposed to be quasi-concave, this leads to
  \[
    f(Px)\geqslant \min_{j=1,\ldots,k}f(y^j)\text{,}
  \]
  which proves the claim.
\end{proof}
Note that in the preceding proof we need \(C\)-monotonicity of \(f\)
on the set \(\O\cap\left(P[S]-C\right)\) only.  If we are given an
objective function \(f^\prime\) and an outer approximation \(\O\)
according to \eqref{assmpt:O} such that \(f^\prime\) is \(C\)-monotone
on \(\O\cap\left(P[S]-C\right)\), we can transform \(f^\prime\) in the
following way in order to obtain a quasi-concave function \(f\) that
complies with \eqref{assmpt:M}:
\[
  f(y)\colonequals
  \begin{cases}
    f^\prime(y)& \text{if \(y\in\O\),}\\
    -\infty&\text{otherwise.}
  \end{cases}
\]

\begin{corollary}
	Let the assumptions \eqref{assmpt:M} and \eqref{assmpt:B} be satisfied. Then \eqref{p} has an optimal solution $x^* \in S$ such that
	\begin{equation*}
		f(Px^*) = \min_{y \in \vertex\P} f(y)\text{.}
	\end{equation*}	
\end{corollary}
\begin{proof}
  By applying Proposition~\ref{prop_exist_minimal} to the set
  $\O = \P$, we know that $\min_{y \in \vertex\P} f(y)$ is a lower
  bound for the optimal value of \eqref{p}.  There are only finitely
  many vertices $y$ of $\P$, each of which can be expressed as
  $y = Px$ for some $x \in S$.  Thus, the lower bound is attained by
  some $x^* \in S$ and $x^*$ is an optimal solution for \eqref{p}.
\end{proof}

\section{Primal algorithm for QCP} \label{sec_prim}

We begin this section by recalling some facts about Benson's algorithm for vector linear programs following the exposition of \cite{benson_type}.
A modified variant of the algorithm is then developed (the results of \cite{linear_multiplicative} are generalized) to solve the quasi-concave scalar optimization problem \eqref{p}. 

\subsection{Primal Benson-type algorithm for VLP}

Benson's algorithm can briefly be described as a procedure computing a
shrinking sequence of outer approximating polyhedra
$\O_j = \{ y \in \R^q \mid B^j y \geq c^j\}$ for $\P$, that is,
\begin{equation*}
  \O_0  \supsetneq \O_1 \supsetneq \dots \supsetneq \O_j  \supsetneq \dots  \supsetneq \O_k = \P\text{.}
\end{equation*}

The procedure is started with $\O_0 \colonequals  \O$ from assumption \eqref{assmpt:O}.
By solving the linear program \eqref{scl:p2} parametrized by an arbitrary vertex $t$ of the current outer approximation $\O_j$, we obtain a boundary point $v$ of $\P$.
An optimal solution of the dual problem \eqref{scl:d2} yields a half-space $\mathcal{H}_j$ supporting $\P$ in $v$, that is, $\mathcal{H}_j \supseteq \P$ and $v \in \P \cap -\mathcal{H}_j$.
The refinement of outer approximations is based on setting
\begin{equation}\label{update_step}
	\O_{j+1} \colonequals  \O_j \cap \mathcal{H}_j\text{.}
\end{equation}
The algorithm terminates when all vertices of $\O_j$ belong to $\P$. 
Benson's algorithm can be seen as a {\em cutting plane method}.
The algorithm as presented in \cite{benson_type} requires the cone $C$ to be solid, i.e. $\inter C \neq \emptyset$.
The following proposition summarizes the role of scalarizations for the algorithm.

\begin{proposition}[{\cite[Proposition
    4.2]{benson_type}}] \label{prop_scal2} Let $S\neq \emptyset$ and
  let assumption \eqref{assmpt:B} be satisfied.  Furthermore, assume
  \(C\) to be solid and let $c \in \inter C$.  Let an H-representation
  of \(C\) be given by
  $C = \{ y \in \mathbb{R}^q \mid Z^{\transpose} y \geq 0\}$.  Then,
  for every $t \in \R^q$, there exist optimal solutions
  $(\bar{x},\bar{z})$ to \eqref{scl:p2} and $(\bar{u},\bar{w})$ to
  \eqref{scl:d2}.  Each solution $(\bar{u},\bar{w})$ to \eqref{scl:d2} defines
  a half-space
  $\mathcal{H}\colonequals \left\{ y \in \R^q \middle |\,
    \bar{w}^\intercal y \geq b^\intercal \bar{u} \right\} \supseteq
  \P$ such that
  $s\colonequals t+c\cdot \bar{z} \in \P \cap -\mathcal{H}$.
  Furthermore, one has
  \begin{equation*}
    t \notin \P \iff \bar{z}>0\text{.}
  \end{equation*}
\end{proposition}

Assumption \eqref{assmpt:O} ensures that an H-representation of the
initial outer polyhedral approximation $\O=\O_0$ of $\P$ is given and
by \eqref{update_step} we obtain iteratively H-representations of all
subsequent outer approximations $\O_j$.  It is necessary to compute
(or to update) a $V$-representation of $\O_j$.  This step is called
{\em vertex enumeration}.  Algorithm \ref{alg_primal_benson} is a
simplified and slightly improved version of the primal Benson
algorithm as formulated in \cite{benson_type}.  In contrast to
\cite[Algorithm 1]{benson_type} we do not store the ``pre-image
information'', i.e.\ $x$ and $(u,w)$.  Moreover, for simplicity, we
store the H-representation of $\O$ directly instead of using duality
theory.  We enhance \cite[Algorithm 1]{benson_type} as we do not
re-initialize the set $T$, and thus avoid solving the same linear
program twice.  The operation $\texttt{solve}(\cdot)$ returns optimal
solutions of a given pair of dual linear programs.

\begin{algorithm}[ht]
  \caption{Simplified version of Benson's algorithm for bounded (i.e.\ assumption \eqref{assmpt:B} holds) \eqref{vlp}, compare \cite[Algorithm 1]{benson_type}.}
  \label{alg_primal_benson}
  \DontPrintSemicolon
  \KwIn{\\Data $A,\; b,\; P,\; Z,\; Y$ (problem data), $c \in \inter C$, $\O$ according to assumption \eqref{assmpt:O}.
  }
  \KwOut{\\ V-representation ($\O\poi,\{\text{columns of Y}\}$) of $\P$ \\ H-representation $\O$ of $\P$}
  \Begin{
    $T \leftarrow \emptyset$ \;
    compute the set $\O\poi$ of vertices of $\O$\;
    \Repeat{$\O\poi \setminus T = \emptyset$}
    { 	
      choose $t \in \O\poi \setminus T$ \;
      $(x,z,u,w)\leftarrow \texttt{solve}(\text{\eqref{scl:p2}}/\text{\eqref{scl:d2}})$\; \label{alg:alg_1_lp}
      \eIf{$z>0$}
      {
        $\O \leftarrow \O \cap \{ y \in \R^q \mid  w^{\transpose} y \geq b^{\transpose} u \}$\;
        update the set $\O\poi$ of vertices of $\O$\;
      }
      {
        $T \leftarrow T \cup \{t\}$ \;
      }
    }
  } 
\end{algorithm}

\begin{theorem}[{see \cite[Theorem 4.5]{benson_type}}]
  Let $S\neq \emptyset$, denote by $C$ a polyhedral convex solid pointed cone
  as defined in \eqref{def_C} which satisfies assumption
  \eqref{assmpt:B}.  Then Algorithm \ref{alg_primal_benson} is correct
  and finite.
\end{theorem}

\subsection{Modified primal Benson-type algorithm for QCP} \label{sec_mod_prim}

We already know that an optimal solution of \eqref{p} can be found if in Algorithm \ref{alg_primal_benson} the vector $x$ with the smallest value $f(Px)$ is stored.
Algorithm \ref{alg_primal_modification} is just a modification and simplification of Algorithm \ref{alg_primal_benson}, which also yields an optimal solution of \eqref{p}.
In general, Algorithm \ref{alg_primal_modification} requires less iteration steps.

\begin{algorithm}[ht]
\DontPrintSemicolon
\SetKwFor{loop}{loop}{}{end}
\KwIn{\\Data $A,\; b,\; P,\; Y,\; Z,\; f$ (problem data), $c \in \inter C$, $\O$ according to assumption \eqref{assmpt:O}}
\KwOut{\\Optimal solution $x$ of (QCP)}
\Begin{
compute the set $\O\poi$ of vertices of $\O$\;
\loop{}
{ 
	choose $t \in \text{argmin} \left\lbrace f(t) \mid  t \in \O\poi \right\rbrace$ \;
	$(x,z,u,w)\leftarrow \texttt{solve}(\text{\eqref{scl:p2}}/\text{\eqref{scl:d2}})$\;
	\eIf{$z > 0$}
	{
		$\O \leftarrow \O \cap \{ y \in \R^q \mid  w^{\transpose} y \geq b^{\transpose} u \}$\;
		update the set $\O\poi$ of vertices of $\O$\;
	}
	{
		break
	}
}

}
\caption{\label{alg_primal_modification}%
  Modified version of Benson's algorithm to solve \eqref{p},
  generalization of \cite[Algorithm 3.2]{linear_multiplicative}}
\end{algorithm}

\begin{theorem} \label{thm:primal_mod}
  Let $S\neq \emptyset$, $f\colon\R^q\to \R$ quasi-concave.  Let $C$
  be a polyhedral convex solid pointed cone according to \eqref{def_C} which
  satisfies assumptions \eqref{assmpt:M} and \eqref{assmpt:B}.  Then
  Algorithm \ref{alg_primal_modification} is correct and finite.
\end{theorem}
\begin{proof}
	The main difference between Algorithms \ref{alg_primal_benson} and \ref{alg_primal_modification} is that Algorithm  \ref{alg_primal_modification} terminates after $z$ equals zero for the first time. Since Algorithm \ref{alg_primal_benson} is finite and as it terminates only if the case $z=0$ occurred at least once, Algorithm \ref{alg_primal_modification} must be finite, too. Thus, it remains to show that $x$ is an optimal solution of \eqref{p}. By Proposition \ref{prop_exist_minimal} and taking into account that $\O\poi = \vertex \O$, we have
	\begin{equation}\label{eq_476}
		f(t) = \min\{\, f(y) \mid  y \in \O\poi\,\} \leq \inf\{\, f(P v) \mid  v \in S\,\}\text{.}
	\end{equation}
	At termination, $z=0$ in \eqref{scl:p2} implies $Px \leq_C t$. Assume that $Px \neq t$. Then there is some $c \in C\setminus\{0\}$ such that $t-c = Px \in \P \subseteq \O$. We also have $t+c \in \O$. This contradicts the fact that $t$ is a vertex of $\O$. Hence $t=Px$. We obtain $f(Px) = f(t)$, where $x$ is feasible for \eqref{p}. Together with \eqref{eq_476} we conclude that $x$ solves \eqref{p}. 
\end{proof}

\begin{example}\label{example_pa}
The following problem is a slight modification of the problem stated and solved in \cite[p.~256]{global_optimization} using polyhedral annexation methods:
\begin{align*}
\min \phantom{platz} &g(x)=- \lvert x_1 \rvert^{\frac{3}{2}} - \frac{1}{10}\left( x_1-0.5x_2+0.3x_3+x4-4.5 \right)^{2}\\[.2cm]
\text{s.t.}  \phantom{platz} &\begin{bmatrix}
1.2 & 1.4 & 0.4 & 0.8 \\
-0.7 & 0.8 & 0.8 & 0.0 \\
0.0 & 1.2 & 0.0 & 0.4 \\
2.8 & -2.1 & 0.5 & 0.0 \\
0.4 & 2.1 & -1.5 & -0.2 \\
-0.6 & -1.3 & 2.4 & 0.5
\end{bmatrix}
x\leq \begin{bmatrix}
6.8 \\ 0.8 \\ 2.1 \\ 1.2 \\ 1.4\\ 0.8
\end{bmatrix},\quad x \geq 0\text{.}\\
\end{align*}
We have $y=Px$ for the matrix 
\begin{align*}
	P=
	\begin{bmatrix}
	1 & 0 & 0& 0 \\
	1 & -0.5 & 0.3 & 1 \\
	\end{bmatrix}
\end{align*}
and the objective function turns into 
\begin{equation*}
	f(y)=-\lvert y_1 \rvert^{\frac{3}{2}}-\frac{1}{10}\left( y_2-4.5 \right)^{2}
\end{equation*}
with $g(x)=f(Px)$.
The feasible region and some level sets of the projected problem are depicted in Figure \ref{level_sets_ex}.
The iteration steps of Algorithm \ref{alg_primal_modification} are shown in Figure \ref{depiction_ex_pa}.
We obtain the optimal value $-2.494$ attained at $(1.084,0.804)$.
\end{example}
\bigskip

\begin{figure}[ht]
\center
\begin{tikzpicture}
	\node[anchor=south west,inner sep=0] (image) at (0,0) {\includegraphics[scale=0.3,trim=1cm 7cm 1cm 8cm]{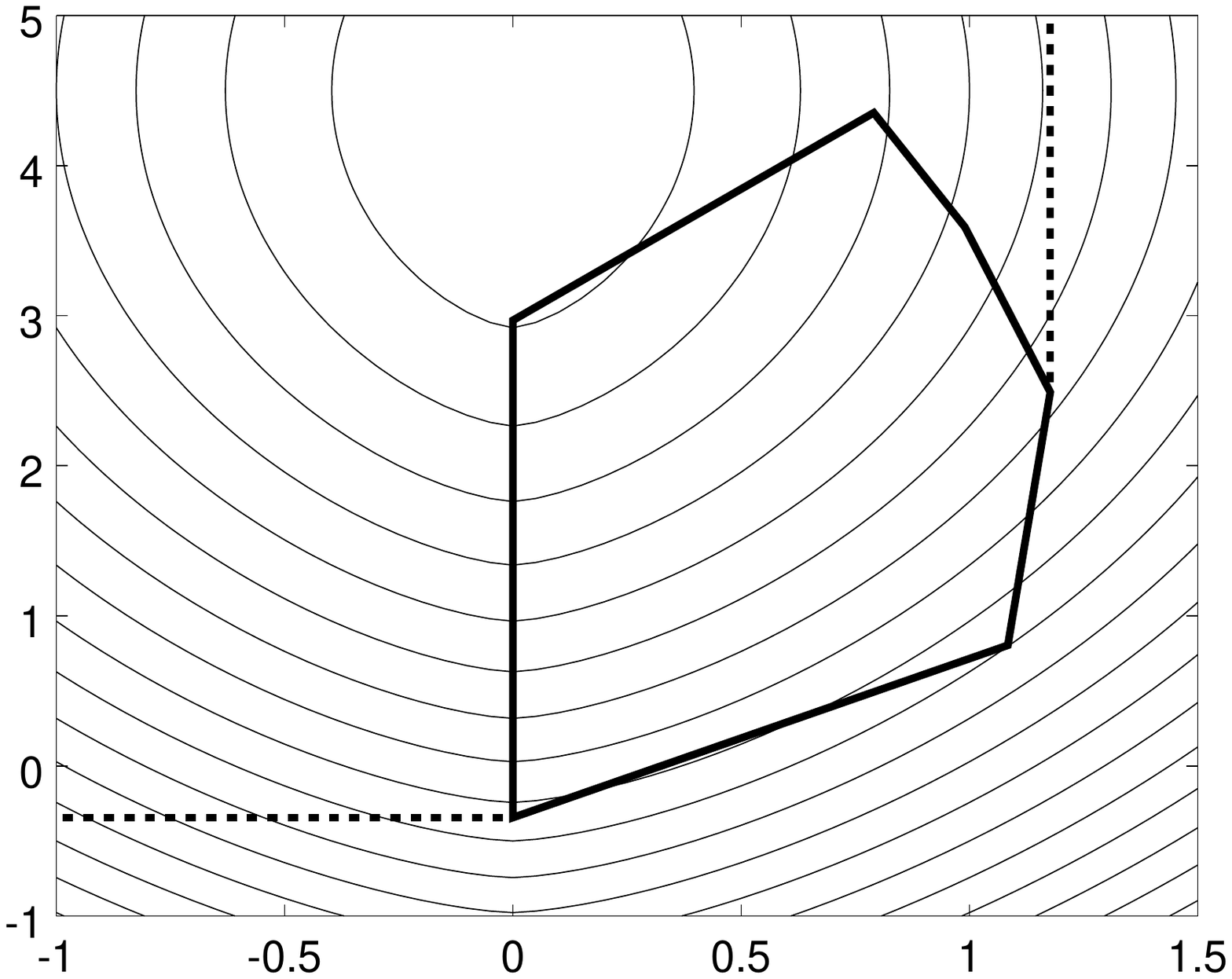}};
	\begin{scope}[
		x={(image.south east)},
		y={(image.north west)}
		]
		\node [black, font=\bfseries] at (0.5,-0.05) {$y_1$};
		\node [black, font=\bfseries] at (0.05,0.5) {$y_2$};
	\end{scope}
\end{tikzpicture}
\caption{\label{level_sets_ex}%
  $P[S]$ and level sets of $f$ for Example \ref{example_pa}. The cone
  $C$ generated by $(-1,0)^\intercal$ and $(0,1)^\intercal$ is
  indicated by the dashed lines. It apparently reflects the
  objective's monotonicity within the feasible region.}
\end{figure}

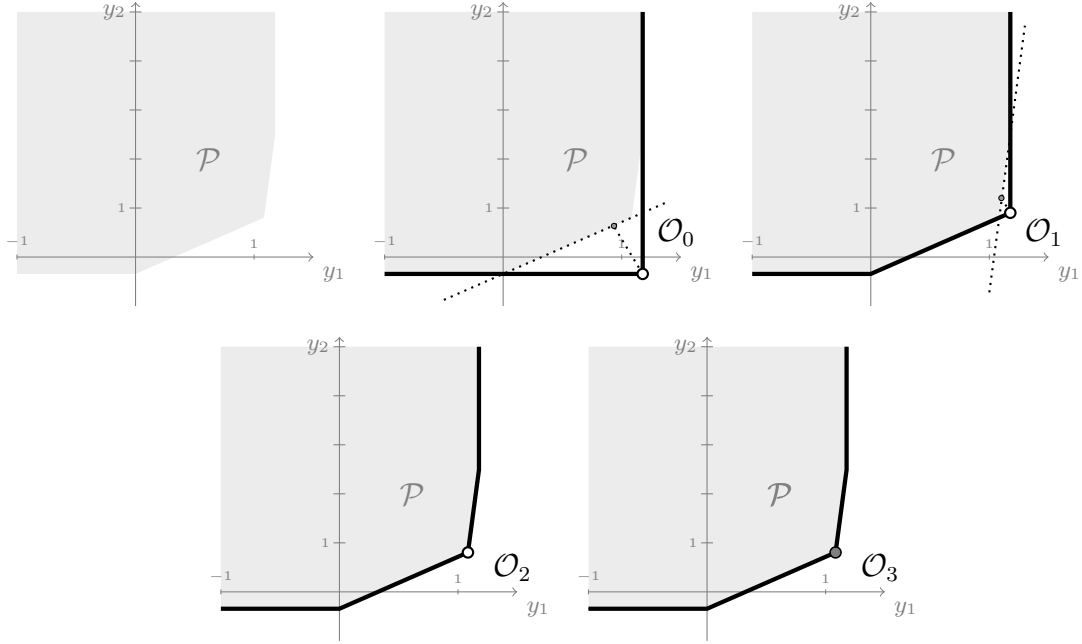
\begin{figure}[ht]
\begin{tikzpicture}[xscale=2.4,scale=0.65]
\path [fill,gray!15!white,ultra thick] (-1,-0.3444) -- (0,-0.3444)--(1.084, 0.804)--(1.177,2.49)--(1.177,5)--(-1,5);
\draw [->,gray] (0,-1) -- (0,5.2) ;
\draw [->,gray] (-1,0) -- (1.5,0) ;
\draw [gray] (-0.05,1)-- (0.05,1);
\node [left,gray] at (0,1) {\tiny $1$};
\node [left,gray] at (0.8,2) {\large $\P$};
\draw [gray] (-0.05,2)-- (0.05,2);
\draw [gray] (-0.05,3)-- (0.05,3);
\draw [gray] (-0.05,4)-- (0.05,4);
\draw [gray] (-0.05,5)-- (0.05,5);
\draw [gray] (-1,-0.05)-- (-1,0.05);
\node [above,gray] at (-1,0) {\tiny $-1$};
\draw [gray] (1,-0.05)-- (1,0.05);
\node [above,gray] at (1,0) {\tiny $1$};
\node [left,gray] at (0,5) {\small $y_2$};
\node [below right,gray] at (1.5,0) {\small $y_1$};
\end{tikzpicture}
\begin{tikzpicture}[xscale=2.4,scale=0.65]
\path [fill,gray!15!white,ultra thick] (-1,-0.3444) -- (0,-0.3444)--(1.084, 0.804)--(1.177,2.49)--(1.177,5)--(-1,5);
\draw [->,gray] (0,-1) -- (0,5.2) ;
\draw [->,gray] (-1,0) -- (1.5,0) ;
\draw [gray] (-0.05,1)-- (0.05,1);
\node [left,gray] at (0,1) {\tiny $1$};
\node [left,gray] at (0.8,2) {\large $\P$};
\draw [gray] (-0.05,2)-- (0.05,2);
\draw [gray] (-0.05,3)-- (0.05,3);
\draw [gray] (-0.05,4)-- (0.05,4);
\draw [gray] (-0.05,5)-- (0.05,5);
\draw [gray] (-1,-0.05)-- (-1,0.05);
\node [above,gray] at (-1,0) {\tiny $-1$};
\draw [gray] (1,-0.05)-- (1,0.05);
\node [above,gray] at (1,0) {\tiny $1$};
\node [left,gray] at (0,5) {\small $y_2$};
\node [below right,gray] at (1.5,0) {\small $y_1$};
\node [left] at (1.7,0.5) {\large $\O_0$};
\draw [dotted, thick, black] (1.177,-0.3444)--(0.927,0.656);
\draw [dotted, thick, black] (-0.5,-0.874)--(1.4,1.139);
\draw [fill=gray!80!white] (0.935,0.634) ellipse (0.625pt and 1.5pt);
\draw [ ultra thick] (-1,-0.3444)--(1.177,-0.3444) -- (1.177,5);
\draw [fill=white,thick] (1.177,-0.3444) ellipse  (1.25pt and 3pt);
\end{tikzpicture}
\begin{tikzpicture}[xscale=2.4,scale=0.65]
\path [fill,gray!15!white,ultra thick] (-1,-0.3444) -- (0,-0.3444)--(1.084, 0.804)--(1.177,2.49)--(1.177,5)--(-1,5);
\draw [->,gray] (0,-1) -- (0,5.2) ;
\draw [->,gray] (-1,0) -- (1.5,0) ;
\draw [gray] (-0.05,1)-- (0.05,1);
\node [left,gray] at (0,1) {\tiny $1$};
\node [left,gray] at (0.8,2) {\large $\P$};
\draw [gray] (-0.05,2)-- (0.05,2);
\draw [gray] (-0.05,3)-- (0.05,3);
\draw [gray] (-0.05,4)-- (0.05,4);
\draw [gray] (-0.05,5)-- (0.05,5);
\draw [gray] (-1,-0.05)-- (-1,0.05);
\node [above,gray] at (-1,0) {\tiny $-1$};
\draw [gray] (1,-0.05)-- (1,0.05);
\node [above,gray] at (1,0) {\tiny $1$};
\node [left,gray] at (0,5) {\small $y_2$};
\node [below right,gray] at (1.5,0) {\small $y_1$};
\node [left] at (1.7,0.5) {\large $\O_1$};
\draw [dotted, thick, black] (1.177,0.9025)--(1.102,1.2025);
\draw [dotted, thick, black] (1,-0.72)--(1.3,4.719);
\draw [fill=gray!80!white]  (1.102,1.2025) ellipse (0.625pt and 1.5pt);
\draw [ultra thick] (-1,-0.3444) -- (0,-0.3444) -- (1.177,0.9025)--(1.177,5);
\draw [thick,fill=white] (1.177,0.9025) ellipse (1.25pt and 3pt);
\end{tikzpicture}
\center
\begin{tikzpicture}[xscale=2.4,scale=0.65]
\path [fill,gray!15!white,ultra thick] (-1,-0.3444) -- (0,-0.3444)--(1.084, 0.804)--(1.177,2.49)--(1.177,5)--(-1,5);
\draw [->,gray] (0,-1) -- (0,5.2) ;
\draw [->,gray] (-1,0) -- (1.5,0) ;
\draw [gray] (-0.05,1)-- (0.05,1);
\node [left,gray] at (0,1) {\tiny $1$};
\node [left,gray] at (0.8,2) {\large $\P$};
\draw [gray] (-0.05,2)-- (0.05,2);
\draw [gray] (-0.05,3)-- (0.05,3);
\draw [gray] (-0.05,4)-- (0.05,4);
\draw [gray] (-0.05,5)-- (0.05,5);
\draw [gray] (-1,-0.05)-- (-1,0.05);
\node [above,gray] at (-1,0) {\tiny $-1$};
\draw [gray] (1,-0.05)-- (1,0.05);
\node [above,gray] at (1,0) {\tiny $1$};
\node [left,gray] at (0,5) {\small $y_2$};
\node [below right,gray] at (1.5,0) {\small $y_1$};
\node [left] at (1.7,0.5) {\large $\O_2$};
\draw [ultra thick] (-1,-0.3444) -- (0,-0.3444)--(1.084, 0.804)--(1.177,2.49)--(1.177,5);

\draw [fill=white,thick] (1.084, 0.804) ellipse (1.25pt and 3pt);

\end{tikzpicture}
\begin{tikzpicture}[xscale=2.4,scale=0.65]
\path [fill,gray!15!white,ultra thick] (-1,-0.3444) -- (0,-0.3444)--(1.084, 0.804)--(1.177,2.49)--(1.177,5)--(-1,5);
\draw [->,gray] (0,-1) -- (0,5.2) ;
\draw [->,gray] (-1,0) -- (1.5,0) ;
\draw [gray] (-0.05,1)-- (0.05,1);
\node [left,gray] at (0,1) {\tiny $1$};
\node [left,gray] at (0.8,2) {\large $\P$};
\draw [gray] (-0.05,2)-- (0.05,2);
\draw [gray] (-0.05,3)-- (0.05,3);
\draw [gray] (-0.05,4)-- (0.05,4);
\draw [gray] (-0.05,5)-- (0.05,5);
\draw [gray] (-1,-0.05)-- (-1,0.05);
\node [above,gray] at (-1,0) {\tiny $-1$};
\draw [gray] (1,-0.05)-- (1,0.05);
\node [above,gray] at (1,0) {\tiny $1$};
\node [left,gray] at (0,5) {\small $y_2$};
\node [below right,gray] at (1.5,0) {\small $y_1$};
\node [left,gray] at (0.8,2) {\large $\P$};
\node [left] at (1.7,0.5) {\large $\O_3$};
\draw [](1.177,5)--(1.177, 2.49)-- (1.084, 0.804)-- (0, -0.3444)-- (-1,-0.344);
\draw [ultra thick] (-1,-0.3444) -- (0,-0.3444)--(1.084, 0.804)--(1.177,2.49)--(1.177,5);
\draw [fill,thick] (1.084, 0.804) ellipse (1.25pt and 3pt);
\draw [gray,fill=gray] (1.084, 0.804) ellipse (1pt and 2.4pt);
\end{tikzpicture}
\caption{\label{depiction_ex_pa}%
  Iteration steps for Example \ref{example_pa}. The white circle shows
  the current point $t$, whereas the gray dots indicate the boundary
  points $s$ calculated by \eqref{scl:p2}. Note that we avoid two
  additional iteration steps in comparison to Algorithm
  \ref{alg_primal_benson} since there are two remaining vertices of
  the outer approximation, which do not need to be processed.}
\end{figure}

\section{Dual algorithm for QCP}\label{sec_dual}

In this section we propose a dual algorithm for the quasi-concave
problem \eqref{p}, which is related to the dual variant of Benson's
algorithm for vector linear programs, introduced in
\cite{dual_variant}.  In \cite{linear_multiplicative},
Shao and Ehrgott introduced a similar algorithm for
linear multiplicative programs, which is a special case of our
setting.  Here, we propose  a modification of
Shao and Ehrgott's algorithm, which turns out to 
yield better numerical results.\par

We start by recalling several facts about duality theory for vector
linear programs.  Afterwards we recapitulate the dual variant of
Benson's algorithm for vector linear programs, which we present in a
simplified form.  The final subsection deals with the dual algorithm
for quasi-concave programs.\par

\subsection{Geometric Duality for VLP} \label{sec_geometric_duality}

The dual problem associated with \eqref{vlp}, introduced in \cite{geom_duality} (see also \cite{buch_andreas,dual_variant}), is
\begin{equation}\label{vlpstar}
	\max\nolimits_K \; D^*(u,w) \quad \text{s.t.}\quad (u,w)\in T \tag{VLP$^*$}
\end{equation}
with objective function $D^* \colon \mathbb{R}^m \times \mathbb{R}^q
\to \mathbb{R}^q$ defined by
\begin{equation*}
  D^*(u,w)=\left( w_1,\ldots,w_{q-1},b^\intercal u  \right)^\transpose\text{,}
\end{equation*}
feasible set
\begin{equation*}
  T\colonequals  \left\{\, (u,w) \in  \mathbb{R}^m \times \mathbb{R}^q \middle |\,   u \geq 0 ,\; A^{\transpose} u=P^{\transpose}w, \; c^\intercal w=1, \; Y^{\transpose} w \geq 0 \,\right\}
\end{equation*}
and ordering cone
\begin{equation*}
	K\colonequals \mathbb{R}_+ \cdot (0,\ldots,0,1)^\transpose\text{.}
\end{equation*}
Throughout, we assume that
\begin{equation}\label{eq_c}
	c \in \text{int }C \quad \text{and} \quad c_q=1\text{.}
\end{equation}
Observe that this assumption does not constitute a restriction:  As
$\inter C \neq \emptyset$, it is always possible to chose $c \in
\inter C$ such that either $c_q=1$ or $c_q=-1$.  In the latter case,
an equivalent problem where $C$, $P$ and $c$ are replaced by $-C$,
$-P$ and $-c$, respectively, can be considered.\par

Similar to the upper image $\P$ for \eqref{vlp}, the lower image for \eqref{vlpstar} is defined as
\begin{equation*}
	\mathcal{D}^*\colonequals D[T]-K\text{.}
\end{equation*}
To express the duality relations, we make use of the following bi-affine {\em coupling function}:
\begin{equation} \label{coupling_func}
  \varphi \colon \mathbb{R}^q \times \mathbb{R}^q \rightarrow \mathbb{R},\quad \varphi(y,y^*)\colonequals \sum_{i=1}^{q-1} y_iy_i^* +y_q \left( 1- \sum_{i=1}^{q-1}c_i y_i^* \right)-y_q^*\text{.}
\end{equation}\

\begin{theorem}[weak duality \cite{geom_duality,buch_andreas}] \label{weak_duality}
  One has
  \begin{equation*}
    \left[\,y \in \P \land y^* \in \mathcal{D}^*\,\right] \implies \varphi(y,y^*)\geq 0\text{.}
  \end{equation*}
\end{theorem}\

\begin{theorem}[strong duality \cite{geom_duality,buch_andreas}] \label{strong_duality}
  Let $S$ and $T$ be nonempty.
  Then
  \begin{align*}
    \left[\,\forall y^* \in \mathcal{D}^* \colon\varphi(y,y^*)\geq 0\,\right] &\implies y \in  \P \\
    \left [\,\forall y \in \P \colon \varphi(y,y^*)\geq 0\,\right] &\implies y^* \in  \mathcal{D}^*\text{.}
  \end{align*}
\end{theorem}\

Using the coupling function $\varphi$ we define half-space-valued functions
\begin{align*}
  \mathcal{H}^* &\colon \mathbb{R}^q \rightrightarrows \mathbb{R}^q &\mathcal{H}^*(y) &\colonequals \left\{ y^* \in \mathbb{R}^q \middle |\,   \varphi(y,y^*) \geq 0 \right\} \\
  \mathcal{H} &\colon \mathbb{R}^q \rightrightarrows \mathbb{R}^q &\mathcal{H}(y^*) &\colonequals \left\{ y \in \mathbb{R}^q \middle |\,   \varphi(y,y^*)\geq 0 \right\}
\end{align*}
and a duality mapping
\begin{equation*}
  \Psi \colon 2^{\mathbb{R}^q} \rightarrow 2^{\mathbb{R}^q}, \quad \Psi(F^*)\colonequals \bigcap_{y^* \in F^*} -\mathcal{H}(y^*) \cap \P\text{.}
\end{equation*}
A proper face $F^*$ of the lower image $\mathcal{D}^*$ is called {\em
  vertical} if $F^*=F^*-K$. Non-vertical proper faces of
$\mathcal{D}^*$ are also called {\em $K$-maximal} as they consist of
$K$-maximal points only.

\begin{theorem}[Geometric Duality \cite{geom_duality}]\label{geometric_duality}
  $\Psi$ is an inclusion reversing one-to-one map (i.e.\ $F_1^*
  \subseteq F_2^* \iff \Psi(F_1^*) \supseteq \Psi(F_2^*)$) between the
  set of all non-vertical proper faces $F^*$ of $\mathcal{D}^*$ and
  the set of all proper faces $F$ of $\P$.
  The inverse map is given by
  \begin{equation*}
    \Psi^{-1} \colon 2^{\mathbb{R}^q} \rightarrow 2^{\mathbb{R}^q}, \quad \Psi^*(F)\colonequals \bigcap_{y \in F}-\mathcal{H}^*(y) \cap \mathcal{D}^* \text{.}
  \end{equation*}
  For non-vertical proper faces $F^*$ of $\mathcal{D}^*$ one has
  \begin{equation*}
    \dim F^* + \dim \Psi(F^*) = q-1\text{.}
  \end{equation*}
\end{theorem}\

In particular, vertices of $\mathcal{D}^*$ correspond to facets of
$\P$ and vertices of $\P$ correspond to non-vertical
facets of $\mathcal{D}^*$.  There is also a correspondence between
the vertical facets of $\mathcal{D}^*$ and extremal directions of
$\P$, see \cite[Section 4.6]{buch_andreas}.

\subsection{Dual variant of Benson's algorithm for VLP}

The dual variant of Benson's algorithm constructs the lower image $\mathcal{D}^*$ by a shrinking sequence of polyhedral outer approximations 
\begin{equation*}
	\O_0^* \supsetneq \O_1^* \supsetneq \dots \supsetneq \O^*_j \supsetneq \dots \supsetneq \O_k^* = \mathcal{D}^*\text{.}
\end{equation*}
This refinement procedure is analogous to the one described in the primal case.
An arbitrary vertex of $\O^*_j$ is either identified as
element of $\mathcal{D}^*$ or is cut off by intersecting
$\O^*_j$ with a suitable half-space $\mathcal{H}^*_j$
obtained from the solution of a scalar problem.
This results in the improved outer approximation $\O^*_{j+1}\colonequals \O^*_j \cap \mathcal{H}^*_j$. 

In order to give a counterpart to Proposition \ref{prop_scal2}, let
$t^*$ be a vertex of $\O_j^*$.  By \[C^+ \colonequals
\left\{y^*\in\R^q\,\middle\vert\,y\in C \Longrightarrow y^\intercal
  y^*\geqslant 0\right\}\] we denote the {\em positive dual} of \(C\).
We set
\begin{align*}
	&\omega(t^*)\colonequals \left( t^*_1,\ldots,t^*_{q-1},1-\sum_{i=1}^{q-1} c_i t^*_i\right)^\transpose
	&\text{and}&
	&\Delta \colonequals  \left\{\, y^* \in \mathbb{R}^q \middle |\,   \omega(y^*) \in C^+  \,\right\}\text{.}
\end{align*}

\begin{proposition}[{\cite[Proposition
    4.6]{benson_type}}] \label{prop_dual_scal}
  Let $S \neq \emptyset$, and let $C$ according to \eqref{def_C} be a solid,
  polyhedral convex pointed cone with $c \in \inter C$ and $c_q = 1$ that
  satisfies assumption~\eqref{assmpt:B}.  Consider $t^* \in \Delta$.
  For $w\colonequals \omega(t^*)$, there exists an optimal solution
  $x$ to \eqref{scl:p1}.  Each solution $x$ to \eqref{scl:p1}
  defines a half-space
  $\mathcal{H}^*(P x)\colonequals \left\{ y^* \in \mathbb{R}^q \, |\,
    \varphi\left( P x,y^* \right)\geq 0 \right\} \supseteq
  \mathcal{D}^*$ such that
  $s^*\colonequals \left( t_1^*,\ldots,t_{q-1}^*,w^\intercal Px
  \right)^\transpose \in -\mathcal{H}^* \cap \mathcal{D}^*$.
  Furthermore, one has $P x \in \bd \P$, and
  \begin{equation*}
    t^* \notin \mathcal{D}^* \iff w^\intercal Px < t_q^*\text{.}
  \end{equation*}
\end{proposition}

\begin{algorithm}[ht]
  \DontPrintSemicolon
  \SetKwFor{loop}{loop}{}{end}
  \KwIn{\\Data $A,\; b,\; P,\;Z,\; Y,\; c$
  }
  \KwOut{\\ V-representation ($\O^*\poi,\{-e^q\}$) of $\mathcal{D}^*$ \\H-representation $\O^*$ of $\mathcal{D}^*$}
  \Begin{		
    $\O^* \leftarrow \{ y^* \in \R^q \mid  y^\intercal \omega(y^*) \geq 0,\; y \in \{\text{columns of } Y\}\}$\; \label{alg3_delta}
    $T^* \leftarrow \emptyset$\; \label{alg3_init1}
    $w \leftarrow \text{sum of columns of $Z$}$\;\label{alg3_init2}
    $t^* \leftarrow (c^\intercal w)^{-1} w$\;\label{alg3_init3}
    $t^*_q \leftarrow \infty$\; \label{alg3_init4}
    \loop{}
    { 
      $w \leftarrow \omega(t^*)$\;	
      $x\leftarrow \texttt{solve}(\text{\eqref{scl:p1}})$\; \label{alg3_lp}
      \eIf{$t^*_q-w^\intercal Px > 0$}
      {
        $\O^* \leftarrow \O^* \cap \{y^* \in \R^q\mid  \varphi(Px,y^*)
        \geq 0\}$\;\label{alg3_cut}
        compute/update the set $\O^*\poi$ of vertices of $\O^*$\;
      }
      {
        $T^* \leftarrow T^* \cup \left\{ t^* \right\}$\;
      }			
      \eIf{$\O^*\poi \setminus T^* \neq \emptyset$\label{alg3_cond}}
      {choose $t^* \in \O^*\poi \setminus T^*$\; \label{alg3_vsr}}
      {break\;}							
    }		
  }
  \caption{\label{alg_dual_benson}%
    Dual variant of Benson's Algorithm, compare \cite[Algorithm
    2]{benson_type}}
\end{algorithm}

Algorithm \ref{alg_dual_benson} is a simplified and slightly modified
version of the dual variant of Benson's algorithm, compare
\cite[Algorithm 2]{benson_type}.  Note that all outer polyhedral
approximations $\O^*$ are contained in the set $\Delta$, compare line
\ref{alg3_delta} of Algorithm \ref{alg_dual_benson}. As the recession
cone of the sets $\O_j^*$ is always $-K$, their V-representations are
already specified by a finite set of points (rather than both points
and directions). Because of the modifications of the algorithm in
comparison to \cite{benson_type}, we sketch the proof of the following
theorem.

\begin{theorem}[{compare \cite[Theorem
    4.9]{benson_type}}]\label{th_alg3}
  Let $S\neq \emptyset$, let $C$ according to \eqref{def_C} be a
  polyhedral convex solid pointed cone with $c \in \inter C $ and $c_q=1$
  such that assumption \eqref{assmpt:B} is satisfied.  Then Algorithm
  \ref{alg_dual_benson} is correct and finite.
\end{theorem}
\begin{proof}
	Since $c \in \inter C$ and $w$ computed in line \ref{alg3_init2} belongs to $C^+ \setminus \{0\}$, we obtain that $c^\intercal w$ in line \ref{alg3_init3} is not zero. The linear program in line \ref{alg3_lp} always has a solution since $t^* \in \Delta$, which holds as the initial set $\O^*$ equals $\Delta$ by line \ref{alg3_delta}. The last component of the normal vector of the half-space $\{y^* \in \R^q\mid  \varphi(Px,y^*) \geq 0\}$ in line \ref{alg3_cut} is $-1$, see \eqref{coupling_func}. Thus, after the first cut in line \ref{alg3_cut} was made, $\O^*$ has a vertex.  It follows that $\O^*\poi$ is always nonempty.
 At termination, we have $\O^*\poi \subseteq T^*$. 
Since $T^* \subseteq \mathcal{D}^*$ and $\O^*\supseteq \mathcal{D}^*$, we obtain $\O^* = \mathcal{D}^*$. \par

To prove that the algorithm is finite, observe that $F^* \colonequals -\mathcal{H}^* \cap \mathcal{D}^*$ in Proposition \ref{prop_dual_scal} is a face of $\mathcal{D}^*$ which belongs to the boundary of $\O^*$ after the cut in line \ref{alg3_cut}.
A vertex $t^*$ of $\O^* \supseteq \mathcal{D}^*$ chosen in a subsequent iteration either belongs to $\mathcal{D}^*$ or cannot belong to the relative interior of $F^*$. In the first case $t^*$ is a vertex of $\mathcal{D}^*$ and is stored in $T^*$. In the second case, another face of $\mathcal{D^*}$ corresponds to the cut. 
Since $\mathcal{D}^*$ has only finitely many faces, the algorithm is finite.
\end{proof}

\subsection{Modified dual variant of Benson's algorithm to solve QCP}

Now the ideas from Section \ref{sec_prim} are applied to modify the dual variant of Bensons's algorithm in order to get a more efficient algorithm for \eqref{p}.
The main difference in comparison to the primal case is that a
shrinking sequence of outer polyhedral approximations $\O_j$
for $\P$ is not part of Algorithm \ref{alg_dual_benson}, but
values of $f$ at vertices of $\O_j$ are required in order to
be able to use the ideas of Section \ref{sec_mod_prim}.
Accepting the computational cost of an extra vertex enumeration step
per iteration allows us to calculate the required sequence of outer
polyhedral approximations $\O_j$ of $\P$.\par

Shao and Ehrgott \cite{linear_multiplicative}
developed a similar algorithm for the special case of multiplicative
linear programs, also see Section~\ref{sec_num}.\par
The main idea of Algorithm \ref{alg_dual_modification} can be
explained as follows: In the loop we compute both, shrinking sequences
of outer approximations $\O_j$ of $\P$ and $\O^*_j$ of $\mathcal{D}^*$. First, in the manner of Algorithm \ref{alg_primal_modification}, a vertex $t$ of $\O$ with minimal value $f(t)$ is selected. Thereafter, a vertex $t^*$ of $\O^*$ which has not yet been identified as a member of $\mathcal{D}^*$ is selected such that $\varphi(t,\cdot)$ is minimal. The difference to Algorithm \ref{alg_dual_benson} is that $t^*$ is selected in this special way. If $\varphi(t,t^*) \geq 0$ for a vertex $t$ of $\O$ with $f$ being minimal and for all vertices $t^*$ of $\O^*$, then $x \in S$ with $t=Px$ is an optimal solution for \eqref{p}. Thus, the selection rule for $t$ and $t^*$ can be motivated as the choice corresponding to the strongest violation of this optimality condition. 

\begin{algorithm}[ht]
	\DontPrintSemicolon 
	\SetKwFor{loop}{loop}{}{end}
	\KwIn{\\Data $A,\; b,\; Z,\;Y,\;c,\; f$, $\O$ according to assumption \eqref{assmpt:O}. }
	\KwOut{\\Solution $x$ of (QCP)}
	\Begin{		
		$\O^* \leftarrow \{ y^* \in \R^q \mid  y^\intercal \omega(y^*) \geq 0,\; y \in \{\text{columns of } Y\}\}$\;
		$T^* \leftarrow \emptyset$\;
		$w \leftarrow \text{sum of columns of $Z$}$\;
		$t^* \leftarrow (c^\intercal w)^{-1} w$\;
		$t^*_q \leftarrow \infty$\;		
		\loop{}
		{ 
			$w \leftarrow \omega(t^*)$\;	
			$x\leftarrow \texttt{solve}(\text{\eqref{scl:p1}})$\; \label{alg_4lp}
			\eIf{$t^*_q-w^\intercal Px > 0$}
			{
				$\O^* \leftarrow \O^* \cap \{y^* \in \R^q\mid  \varphi(Px,y^*)
				 \geq 0\}$\;\label{alg4_cut}
				compute/update the set $\O^*\poi$ of vertices of $\O^*$\;
			}
			{
				$T^* \leftarrow T^* \cup \left\{ t^* \right\}$\;
			}
			$\O \leftarrow \O \cap \{ y \in \R^q \mid  w^\intercal y \geq w^\intercal Px \}$\;	\label{alg4_add_start}
			compute/update the set $\O\poi$ of vertices of $\O$ \;
			choose $t \in \argmin \left\lbrace f(y) \mid  y \in \O\poi \right\rbrace$\; \label{alg4_argmin}\label{alg4_add_end}	
			\eIf{$\min \left\lbrace   \varphi(t,y^*) \mid   y^* \in \O^*\poi \right\rbrace < 0$\label{alg4_cond}}{ 
			choose $t^* \in \argmin \left\lbrace \varphi(t,y^*) \mid  y^* \in \O^*\poi \setminus T^*\right\rbrace$\; \label{alg4_vsr}}
			{determine $x \in S$ such that $t=Px$ and break\;\label{alg4_stop}}
		}		
	}
	\caption{\label{alg_dual_modification}%
          Modified dual variant of Benson's algorithms
	}
	
\end{algorithm}

\begin{theorem}
  Let $S \neq \emptyset$, $f \colon \R^q \rightarrow \extR$
  quasi-concave, and let
  $C= \left\{ y \in \mathbb{R}^q \,\middle\vert\, Z^{\transpose}y \geq
    0 \right\}$ be a polyhedral convex solid pointed cone which satisfies
  assumptions \eqref{assmpt:M} and \eqref{assmpt:B}.  Then Algorithm
  \ref{alg_dual_modification} is correct and finite.
\end{theorem}
\begin{proof}
  Note that Algorithms \ref{alg_dual_modification} coincides with Algorithm \ref{alg_dual_benson} up to the following changes: 
  \begin{enumerate}[(a)]
  \item the additional lines
    \ref{alg4_add_start}--\ref{alg4_add_end} to compute
    the outer approximations $\O$ of $\P$,
  \item  a different stopping condition for the loop in line \ref{alg4_cond},
  \item a specific rule to select some $t^*$ from the set $\O^*\poi \setminus T^*$ in line \ref{alg4_vsr},
  \item the computation of the result $x$ in line \ref{alg4_stop}.
  \end{enumerate}  
  Therefore, the result follows from Theorem \ref{th_alg3} by taking into account the following facts:	
  \begin{enumerate}[(a)]
  \item The new lines
    \ref{alg4_add_start}--\ref{alg4_add_end} are well defined,
    in particular, by assumption \eqref{assmpt:O}, the set $\O$
    has a vertex.
    
  \item We show that the condition
    \begin{equation}\label{eq_cond4}
      \min \left\lbrace   \varphi(t,y^*) \mid   y^* \in \O^*\poi \right\rbrace < 0
    \end{equation}	
    in Algorithm \ref{alg_dual_modification}, line
    \ref{alg4_cond} implies the corresponding condition
    \begin{equation}\label{eq_cond3}
      \O^*\poi \setminus T^* \neq \emptyset
    \end{equation}
    in Algorithm \ref{alg_dual_benson}, line
    \ref{alg3_cond}. Assume that \eqref{eq_cond4} is satisfied
    but \eqref{eq_cond3} is violated, i.e.,
    $\O^*\poi \subseteq T^*$. Since
    $T^* \subseteq \mathcal{D}^*$ and
    $\O^*\supseteq \mathcal{D}^*$, we obtain
    $\O^* = \mathcal{D}^*$. By construction, we have
    \begin{equation*}
      \O \subseteq \{y \in \R^q\mid  \forall t^* \in T^*:\; \omega(t^*)^\transpose y \geq t^*_q\} \subseteq \{y \in \R^q\mid  \forall t^* \in  \text{vert}\,\mathcal{D}^*:\;\varphi(y,t^*)\geq 0\}  = \P\text{,}
    \end{equation*}
    where the last equation follows from geometric duality, see
    Theorem \ref{geometric_duality}.  On the other hand, we have
    $\O \supseteq \P$, whence $t \in \O= \P$. Weak duality
    (Theorem \ref{weak_duality}) implies that
    $\varphi(t,t^*)\geq 0$ for all $t^* \in \O^*\poi$, which
    contradicts \eqref{eq_cond4}. This proves that
    \eqref{eq_cond4} implies \eqref{eq_cond3}. Hence, from the
    finiteness of Algorithm \ref{alg_dual_benson} finiteness of
    Algorithm \ref{alg_dual_modification} follows.

  \item As shown in (b), the specific choice of $t^*$ in line
    \ref{alg4_vsr} is well-defined.

  \item At termination, we have $\varphi(t,y^*) \geq 0$ for all
    $y^* \in \O^*\poi$. As $\O^* = \conv \O^*\poi -K$, the
    inequality also holds for all
    $y^* \in \mathcal{D}^* \subseteq \O^*$. Theorem
    \ref{strong_duality} implies $t \in \P$. Since $t \in \P$ is
    a vertex of $\O$ and $\O\supseteq \P$, $t$ is also a vertex
    of $\P$.  Taking into account that $\P = P[S] + C$, we
    conclude that there exists $x \in S$ with $t =Px$. Since $t$
    in line \ref{alg4_argmin} was chosen from $\O\poi$ such that
    $f(t)$ is minimal, Proposition \ref{prop_exist_minimal}
    yields that $x$ is an optimal solution of \eqref{p}.
  \end{enumerate}
\end{proof}

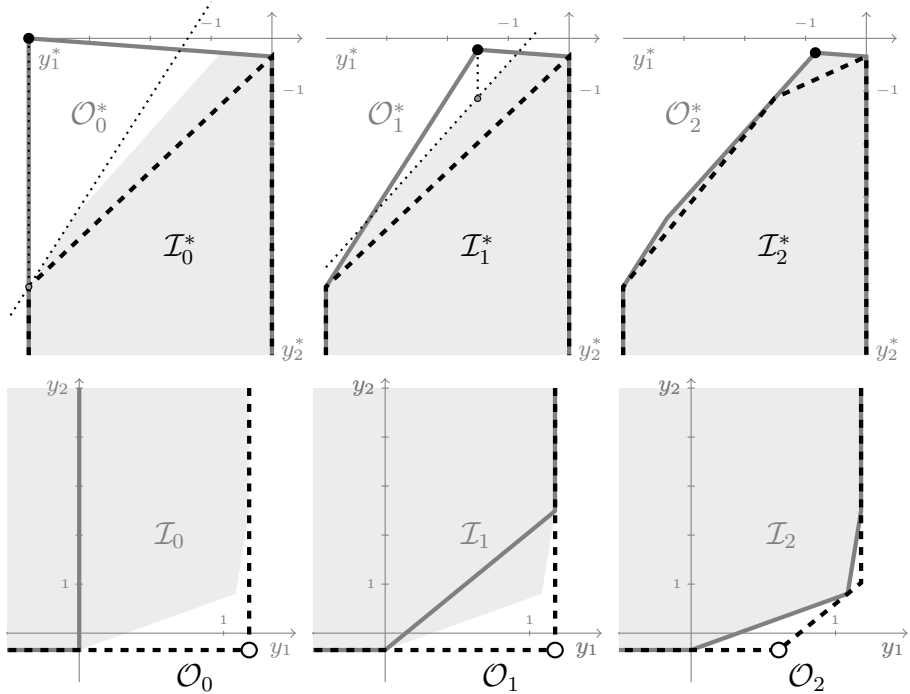
\begin{figure}[ht]
\center
\begin{tikzpicture}[xscale=0.8,yscale=0.7]
\path [fill,gray!15!white,ultra thin] (-4,-6) -- (-4,-4.707)--(-3.277,-3.406)--(-0.838,-0.272)--(0,-0.3444)--(0,-6)--(-4,-6);
\draw [->,gray] (0,-6) -- (0,0.5) ;
\draw [->,gray] (-4,0) -- (0.5,0) ;
\draw [gray](-0.05,-1)-- (0.05,-1);
\node [right,gray] at (0,-1) {\tiny $-1$};
\draw [gray] (-0.05,-2)-- (0.05,-2);
\draw [gray](-0.05,-3)-- (0.05,-3);
\draw [gray](-0.05,-4)-- (0.05,-4);
\draw [gray](-1,-0.05)-- (-1,0.05);
\node [above,gray] at (-1,0) {\tiny $-1$};
\draw [gray](-2,-0.05)-- (-2,0.05);
\draw [gray](-3,-0.05)-- (-3,0.05);
\node [below right,gray] at (0,-5.5) {\small $y_2^*$};
\node [below right,gray] at (-4,0) {\small $y_1^*$};
\node [gray] at (-3,-1.5) {\large $\O^*_0$};
\node [] at (-1.5,-4) {\large $\mathcal{I}^*_0$};
\draw [gray,ultra thick] (-4,-6) -- (-4,0)--(0,-0.344)--(0,-6);
\draw [ dashed,ultra thick] (-4,-6) -- (-4,-4.707)--(0,-0.344)--(0,-6);
\draw [dotted,thick] (-4,0)--(-4,-4.707);
\draw [dotted,thick] (-4.3,2.489+1.799*-4.3)--(-1,2.489+1.799*-1);
\draw [thick,fill=black] (-4,0) ellipse (2.1pt and 2.4pt);
\draw [fill=gray!80!white] (-4,-4.707) ellipse (1.4pt and 1.6pt);

\end{tikzpicture}
\begin{tikzpicture}[xscale=0.8,yscale=0.7]
\path [fill,gray!15!white,ultra thin] (-4,-6) -- (-4,-4.707)--(-3.277,-3.406)--(-0.838,-0.272)--(0,-0.3444)--(0,-6)--(-4,-6);
\draw [->,gray] (0,-6) -- (0,0.5) ;
\draw [->,gray] (-4,0) -- (0.5,0) ;
\draw [gray](-0.05,-1)-- (0.05,-1);
\node [right,gray] at (0,-1) {\tiny $-1$};
\draw [gray] (-0.05,-2)-- (0.05,-2);
\draw [gray](-0.05,-3)-- (0.05,-3);
\draw [gray](-0.05,-4)-- (0.05,-4);
\draw [gray](-1,-0.05)-- (-1,0.05);
\node [above,gray] at (-1,0) {\tiny $-1$};
\draw [gray](-2,-0.05)-- (-2,0.05);
\draw [gray](-3,-0.05)-- (-3,0.05);
\node [below right,gray] at (0,-5.5) {\small $y_2^*$};
\node [below right,gray] at (-4,0) {\small $y_1^*$};
\node [gray] at (-3,-1.5) {\large $\O^*_1$};
\node [] at (-1.5,-4) {\large $\mathcal{I}^*_1$};
.\draw [gray,ultra thick](-4,-6) -- (-4,-4.707)--(-1.504,-0.215)--(0,-0.344)--(0,-6);
\draw [ dashed,ultra thick] (-4,-6) -- (-4,-4.707)--(0,-0.344)--(0,-6);
\draw [dotted,thick] (-1.504,-0.215)--(-1.504,-1.14);
\draw [dotted,thick] (-4,0.815+1.288*-4)--(-0.5,0.815+1.288*-0.5);
\draw [thick,fill=black] (-1.504,-0.215) ellipse (2.1pt and 2.4pt);
\draw [fill=gray!80!white] (-1.504,-1.14) ellipse (1.4pt and 1.6pt);
\end{tikzpicture}
\begin{tikzpicture}[xscale=0.8,yscale=0.7]
\path [fill,gray!15!white,ultra thin] (-4,-6) -- (-4,-4.707)--(-3.277,-3.406)--(-0.838,-0.272)--(0,-0.3444)--(0,-6)--(-4,-6);
\draw [->,gray] (0,-6) -- (0,0.5) ;
\draw [->,gray] (-4,0) -- (0.5,0) ;
\draw [gray](-0.05,-1)-- (0.05,-1);
\node [right,gray] at (0,-1) {\tiny $-1$};
\draw [gray] (-0.05,-2)-- (0.05,-2);
\draw [gray](-0.05,-3)-- (0.05,-3);
\draw [gray](-0.05,-4)-- (0.05,-4);
\draw [gray](-1,-0.05)-- (-1,0.05);
\node [above,gray] at (-1,0) {\tiny $-1$};
\draw [gray](-2,-0.05)-- (-2,0.05);
\draw [gray](-3,-0.05)-- (-3,0.05);
\node [below right,gray] at (0,-5.5) {\small $y_2^*$};
\node [below right,gray] at (-4,0) {\small $y_1^*$};
\node [gray] at (-3,-1.5) {\large $\O^*_2$};
\node [] at (-1.5,-4) {\large $\mathcal{I}^*_2$};
\draw [gray,ultra thick] (-4,-6) -- (-4,-4.707)--(-3.277,-3.406)--(-0.838,-0.272)--(0,-0.3444)--(0,-6);
\draw [ dashed,ultra thick] (-4,-6) -- (-4,-4.707)--(-1.503,-1.127)--(0,-0.344)--(0,-6);
\draw [thick,fill=black] (-0.838,-0.272) ellipse (2.1pt and 2.4pt);
\end{tikzpicture}

\begin{tikzpicture}[xscale=1.9,yscale=0.65]
\path [fill,gray!15!white,dashed , ultra thick] (-0.5, -0.3444)--(0, -0.344)--(1.084, 0.804)--(1.177, 2.49)--(1.177, 5)--(-0.5,5);
\draw [->,gray] (0,-1) -- (0,5.2) ;
\draw [->,gray] (-0.5,0) -- (1.5,0) ;
\draw [gray](-0.025,1)-- (0.025,1);
\node [left,gray] at (0,1) {\tiny $1$};
\draw [gray](-0.025,2)-- (0.025,2);
\draw [gray](-0.025,3)-- (0.025,3);
\draw [gray](-0.025,4)-- (0.025,4);
\draw [gray](-0.025,5)-- (0.025,5);
\draw [gray](-0.5,-0.025)-- (-0.5,0.025);
\node [above,gray] at (1,0) {\tiny $1$};
\draw [gray](1,-0.05)-- (1,0.05);
\node [left,gray] at (0,5) {\small $y_2$};
\node [below,gray] at (1.4,0) {\small $y_1$};
\node [left,gray] at (0.8,2) {\large $\mathcal{I}_0$};
\node [left] at (1,-0.9) {\large $\O_0$};
\draw [gray,ultra thick] (-0.5, -0.3444)--(0, -0.344)--(0, 5);
\draw [ultra thick,dashed]  (1.177,5)--(1.177,-0.3444)--(-0.5,-0.3444);
\draw [fill=white,thick] (1.177,-0.3444) ellipse (1.5pt and 4.25pt);
\end{tikzpicture}
\begin{tikzpicture}[xscale=1.9,yscale=0.65]
\path [fill,gray!15!white,dashed , ultra thick] (-0.5, -0.3444)--(0, -0.344)--(1.084, 0.804)--(1.177, 2.49)--(1.177, 5)--(-0.5,5);
\draw [->,gray] (0,-1) -- (0,5.2) ;
\draw [->,gray] (-0.5,0) -- (1.5,0) ;
\draw [gray](-0.025,1)-- (0.025,1);
\node [left,gray] at (0,1) {\tiny $1$};
\draw [gray](-0.025,2)-- (0.025,2);
\draw [gray](-0.025,3)-- (0.025,3);
\draw [gray](-0.025,4)-- (0.025,4);
\draw [gray](-0.025,5)-- (0.025,5);
\draw [gray](-0.5,-0.025)-- (-0.5,0.025);
\node [above,gray] at (1,0) {\tiny $1$};
\draw [gray](1,-0.05)-- (1,0.05);
\node [left,gray] at (0,5) {\small $y_2$};
\node [below,gray] at (1.4,0) {\small $y_1$};
\node [left,gray] at (0,5) {\small $y_2$};
\node [below,gray] at (1.4,0) {\small $y_1$};
\node [left,gray] at (0.8,2) {\large $\mathcal{I}_1$};
\node [left] at (1,-0.9) {\large $\O_1$};
\draw [gray,ultra thick] (-0.5, -0.3444)--(0, -0.344)--(1.176,2.49)--(1.176, 5);
\draw [ultra thick,dashed]  (1.177,5)--(1.177,-0.3444)--(-0.5,-0.3444);
\draw [fill=white,thick] (1.177,-0.3444) ellipse (1.5pt and 4.25pt);
\end{tikzpicture}
\begin{tikzpicture}[xscale=1.9,yscale=0.65]
\path [fill,gray!15!white,dashed , ultra thick] (-0.5, -0.3444)--(0, -0.344)--(1.084, 0.804)--(1.177, 2.49)--(1.177, 5)--(-0.5,5);
\draw [->,gray] (0,-1) -- (0,5.2) ;
\draw [->,gray] (-0.5,0) -- (1.5,0) ;
\draw [gray](-0.025,1)-- (0.025,1);
\node [left,gray] at (0,1) {\tiny $1$};
\draw [gray](-0.025,2)-- (0.025,2);
\draw [gray](-0.025,3)-- (0.025,3);
\draw [gray](-0.025,4)-- (0.025,4);
\draw [gray](-0.025,5)-- (0.025,5);
\draw [gray](-0.5,-0.025)-- (-0.5,0.025);
\node [above,gray] at (1,0) {\tiny $1$};
\draw [gray](1,-0.05)-- (1,0.05);
\node [left,gray] at (0,5) {\small $y_2$};
\node [below,gray] at (1.4,0) {\small $y_1$};
\node [left,gray] at (0,5) {\small $y_2$};
\node [below,gray] at (1.4,0) {\small $y_1$};
\node [left,gray] at (0.8,2) {\large $\mathcal{I}_2$};
\node [left] at (1,-0.9) {\large $\O_2$};
\draw [gray , ultra thick] (-0.5, -0.3444)--(0, -0.344)--(1.084, 0.804)--(1.177, 2.49)--(1.177, 5);
\draw [ultra thick,dashed] (1.177,5) -- (1.177,1.027)--(0.607,-0.3444)--(-0.5,-0.3444);
\draw [thick,fill=white] (0.607,-0.3444) ellipse (1.5pt and 4.25pt);
\end{tikzpicture}

\caption{\label{depiction_ex_pa_dual}%
  Visualization of the initialization and the first two iteration
  steps of Algorithm \ref{alg_dual_modification} by Example
  \ref{example_pa_dual}.  The shrinking sequence $\O^*_j$ of outer
  approximations of $\mathcal{D}^*$ corresponds to the expanding
  sequence
  $\mathcal{I}_j \colonequals \{ y \in \R^q \mid \forall y^* \in
  \O^*_j:\; \varphi(y,y^*)\geq 0\}$ of inner approximations of $\P$ by
  geometric duality. Likewise, there is an expanding sequence
  $\mathcal{I}_j^* \colonequals \{ y^* \in \R^q \mid \forall y \in
  \O_j:\; \varphi(y,y^*)\geq 0\}$ of inner approximations of
  $\mathcal{D}^*$ corresponding to the shrinking sequence of outer
  approximations $\O_j$ of $\P$. The white circle indicates the vertex
  of $\O_j$ chosen as $t$ in line \ref{alg4_argmin}. The black dots
  label the corresponding points $t^*$, see line
  \ref{alg4_vsr}. Again, the gray dots indicate the calculated
  boundary points.  The calculations are based on the choice of
  $c=(-0.25 , 1)^\intercal $ as inner point of $C$. Notice that even
  though we have $\mathcal{D}^*=\O^*_2$, the algorithm does not
  terminate after two iterations because $t$ is not an element of
  $\P$. Another two iteration steps are required to identify the
  problems solution $(1.084,0.804)^\intercal$.}
\end{figure}

\begin{example} \label{example_pa_dual}
	Consider the problem stated in Example  \ref{example_pa}. The first steps of Algorithm \ref{alg_dual_modification} are shown in Figure \ref{depiction_ex_pa_dual}.
\end{example}

\section{Extension to the case of non-solid cones}  \label{nonsolid_cones}

The algorithms developed in Sections \ref{sec_prim} and \ref{sec_dual} are based on the assumption that the cone $C$ has a nonempty interior. Some $c \in \inter C$ is required in Proposition \ref{prop_scal2} and for the duality results in Section \ref{sec_geometric_duality}. 

As shown in \cite{poly_proj}, any vector linear program 
\begin{equation*}\tag{VLP}
\min\nolimits_C \;Px \quad \text{s.t.} \quad Ax \geq b \text{,}
\end{equation*}
for $C=\{y \in \R^q \mid  Z^{\transpose} y \geq 0\}$, can be reformulated as a vector linear program 
\begin{equation}\label{vlp1}
\min\nolimits_{\mathbb{R}^{q+1}_+} \begin{pmatrix}y\\-e^\intercal y \end{pmatrix} \quad \text{s.t.} \quad Z^{\transpose} y \geq Z ^{\transpose} Px,\,Ax \geq b \tag{VLP'}\text{.}
\end{equation}
We use \(e\) to denote the vector whose entries are all equal to one
of appropriate dimension.  Observe that the ordering cone used in
\eqref{vlp1} is the non-negative orthant $\mathbb{R}^{q+1}_+$.  The
relationship between the upper images $\P$ of \eqref{vlp} and
$\mathcal{M}$ of \eqref{vlp1} can be described as follows. Consider
the hyperplane
$H\colonequals \{y \in \R^{q+1} \mid e^{\transpose} y = 0\}$ and let
$\pi : \R^{q+1} \to \R^q$ be the projection defined by
$\pi(y_1,\dots,y_q,y_{q+1}) \colonequals (y_1,\dots,y_q)$
(i.e. cancellation of the last component). Then
\begin{equation*}
	 \P = \pi(\mathcal{M} \cap H)\text{.}
\end{equation*}
The next result shows that assumption \eqref{assmpt:B} is not satisfied for \eqref{vlp1} in most cases.
\begin{proposition}
	Let $C\neq \{0\}$. If Problem \eqref{vlp1} is feasible, then it is unbounded.
\end{proposition}
\begin{proof}
Since $C \neq \{0\}$, there is a nonzero vector $c \in C$. 
Consider the vector 
\begin{equation*}
\bar{c}\colonequals \begin{pmatrix}
c \\ -e^\intercal c
\end{pmatrix} \in \R^{q+1}\text{.}
\end{equation*}
Let $\bar z$ be an arbitrary point in the upper image 
\begin{equation*}
	\mathcal{M} = \left\{ \begin{pmatrix}z \\ \zeta \end{pmatrix} \middle |\,   \exists y \in \P\colon z \geq y,\; \zeta \geq -e^\intercal y \right\}
\end{equation*}
of \eqref{vlp1} and $\lambda \geq 0$. Then $\bar z + \lambda \bar c \in \mathcal{M}$, i.e., $\bar c$ is a direction of $\mathcal{M}$. But obviously $\bar c$ does not belong to $\R^{q+1}_+$. Thus, $\eqref{vlp1}$ is unbounded.
\end{proof}
This problem can be solved by enlarging the ordering cone $\R^{q+1}_+$ appropriately. Let $Y \in \R^{q\times o}$ denote a matrix whose columns are generators of the cone $C$, that is, $C=\{ Y \mu \mid  \mu \geq 0 \}$. Set
\begin{equation*}
R \colonequals  \left\{ \left( I ,\begin{matrix} Y\\ -e^\intercal Y \end{matrix} \right) \begin{pmatrix}\lambda\\ \mu \end{pmatrix} \; \vert \; \lambda \in \mathbb{R}^{q+1}_+ ,\, \mu \in \mathbb{R}^{o}_+   \right\}
\end{equation*}
and consider the problem
\begin{equation}\label{vlp2}
\min\nolimits_R \begin{pmatrix} Px \\-e^\intercal Px \end{pmatrix} \quad \text{s.t.} \quad Ax \geq b \tag{VLP''}\text{.}
\end{equation}
Problems \eqref{vlp} and \eqref{vlp2} are related in the following sense.

\begin{proposition}
  Let $\P$ be the upper image of \eqref{vlp} and let $\mathcal{M}$ be the upper image of \eqref{vlp2}. Then
  \begin{equation*}
    \P = \pi(\mathcal{M} \cap H)\text{.}
  \end{equation*}
\end{proposition}
\begin{proof}
  From the problem definitions we have 
  \begin{align*}
    \mathcal{M} &= \left\{\begin{pmatrix}
        z \\ \zeta\end{pmatrix}  \middle\vert\; \exists x \in S\colon\begin{pmatrix}
        z \\ \zeta
      \end{pmatrix} \geq_R \begin{pmatrix}
        Px \\ -e^\intercal Px
      \end{pmatrix} \right\} \\
                &= \left\{ \begin{pmatrix}
                    z \\ \zeta\end{pmatrix}  \middle\vert\;  \exists x \in S,\; \exists \mu \in \R^{o}_+\colon\begin{pmatrix}
                    z \\ \zeta
                  \end{pmatrix} 
    - \begin{pmatrix}
      Px \\ -e^\intercal Px
    \end{pmatrix} \geq \begin{pmatrix}
      Y \\ -e^\intercal Y 
    \end{pmatrix} \mu \right\}\text{.}	
  \end{align*}
  Thus, we have
  \begin{align*}
    \mathcal{M} \cap H 
    &= \left\{ \begin{pmatrix}
        z \\ -e^\intercal z \end{pmatrix}  \middle\vert\;  \exists x \in S,\; \exists \mu \in \R^{o}_+\colon\begin{pmatrix}
        z- Px \\ -e^\intercal (z-Px)
      \end{pmatrix} \geq \begin{pmatrix}
        Y \mu \\ -e^\intercal Y \mu 
      \end{pmatrix} \right\}\\
    &=
      \left\{ \begin{pmatrix}
          z \\ -e^\intercal z \end{pmatrix}  \middle\vert\;  \exists x \in S,\; \exists \mu \in \R^{o}_+\colon 
    z- Px = Y \mu \right\}\text{,}
  \end{align*}
  which implies the claim.
\end{proof}

Let Problem \eqref{p} as defined in Section \ref{sec_approach}, in
particular, let a quasi-concave function $f\colon\R^q \to \extR$, be given and let
$C \subseteq \R^q$ be a polyhedral convex pointed cone such that the
assumptions \eqref{assmpt:M} and \eqref{assmpt:B} are satisfied. We
define
\begin{equation*}
  \bar{f} \colon \R^{q+1} \to \R,\quad \bar{f}
  \left(
    \begin{pmatrix}
      y\\
      \eta
    \end{pmatrix}
  \right)\colonequals
  \begin{cases}
    f(y)& \text{if } e^\intercal y +\eta \geqslant 0\text{,}\\
    -\infty&\text{otherwise.}
  \end{cases}
\end{equation*}
As $f$ is quasi-concave on \(\R^q\), $\bar{f}$ is quasi-concave on
\(\R^{q+1}\).  Further, we define
\[
  \bar{P}\colonequals
  \begin{pmatrix}
    P\\
    -e^\intercal P
  \end{pmatrix}\text{.}
\]
The following quasi-concave problem is a reformulation
of \eqref{p} with the same optimal solution:
\begin{equation}\label{p1}\tag{QCP'}
  \min \bar{f}(\bar{P}x)\quad \text{s.t.} \quad A x \geq b\text{.}
\end{equation}
The associated vector linear program is \eqref{vlp2}.
\begin{proposition}
  Let \eqref{assmpt:M} be satisfied for \eqref{p} and cone \(C\).
  Then assumption \eqref{assmpt:M} does also hold for \eqref{p1} with
  respect to the cone \(R\), i.e.\ \(\bar{f}\) is \(R\)-monotone on
  the set \(\bar{P}[S] - R\).
\end{proposition}
\begin{proof}
  Consider \(z^1\in\bar{P}[S]-R\), i.e.
  \[
    z^1 =
    \begin{pmatrix}
      y - c^1\\
      -e^\intercal (y -c^1)
    \end{pmatrix}
	- \sigma^1
  \]
  for \(y\in P[S]\) and some \(\sigma^1\in\R^{q+1}_+, c^1\in C\). Furthermore,
  let \(z^2\in\R^{q+1}\) with \(z^2\leqslant_R z^1\), meaning
  \begin{align*}
    z^2 &= z^1 -
    \begin{pmatrix}
      c^2\\
      -e^\intercal c^2
    \end{pmatrix}
	-\sigma^2\\
	&=
	\begin{pmatrix}
		y - c^1 - c^2\\
		-e^\intercal\left(y - c^1 - c^2\right)
	\end{pmatrix}
	-\left(\sigma^1 + \sigma^2\right)
  \end{align*}
  for some \(\sigma^2\in\R^{q+1}_+,c^2\in C\), be given.  Then
  \[
  	\bar{f}(z^2) =
	\begin{cases}
		f(y - c^1 - c^2) & \text{if }\sigma^1 = \sigma^2 = 0\\
		-\infty &\text{otherwise}
	\end{cases}
  \]
  holds. In the case \(\sigma^j = 0\) for \(j=1,2\), due to condition
  \eqref{assmpt:M} for \eqref{p}, we get
  \[
  \bar{f}(z^2)=f(y - c^1 - c^2) \leqslant f(y - c^1) = \bar{f}(z^1)\text{,}
  \]
  which proves the claim.
\end{proof}

\begin{proposition}
	Let \eqref{assmpt:B} be satisfied for Problem \eqref{p} with cone \(C\).
Then \eqref{assmpt:B} does also hold for \eqref{p1}, i.e.\ \(\bar{P}[S]\) is
bounded with respect to \(R\).
\end{proposition}
\begin{proof}
Let $P[S]$ be $C$-bounded, i.e.\ it holds $0^+P[S] \subseteq C$. The claimed  statement immediately follows from
\begin{align*}
0^+\bar{P}[S] = \begin{pmatrix}
0^+P[S] \\ -e^\intercal \left( 0^+P[S] \right)
\end{pmatrix} \subseteq \begin{pmatrix}
C \\ -e^\intercal C
\end{pmatrix} \subseteq R \, .
\end{align*}  
\end{proof}

Let us summarize the results.
\begin{corollary}
	The assumption $\inter C \neq \emptyset$ can be dropped, when Algorithm \ref{alg_primal_modification} or Algorithm \ref{alg_dual_modification} is applied to the reformulated quasi-concave problem \eqref{p1}.
\end{corollary}

For illustration reasons we close this section with an example.

\begin{example} \label{ex_nonsolid}
Consider the concave problem 
\begin{equation*}
\min \; y_1 -y_2^2 \quad
\text{s.t.}\quad  x\in S , \; y=Px \text{.}
\end{equation*} 
for a matrix $P \in \mathbb{R}^{2 \times n}$. Without any information
about the structure of $P[S]$ and due to the quadratic impact of
$y_2$, the largest polyhedral monotonicity cone usable is $C=(1,
0)^\transpose  \cdot \mathbb{R}_+$. This cone is obviously non-solid in
$\mathbb{R}^2$. To illustrate the method discussed above we
consider the problem
\begin{equation*}
\min \; y_1 -y_2^2 \quad\text{s.t.}\quad  -e \leq x,\; x_1\leq 1,\; x_3 \leq 1, \; y=Px\text{,} 
\end{equation*}
where we set 
\begin{align*}
	P=\begin{pmatrix} 1 & 1 &-1 \\ 1 & 0 & 1\end{pmatrix}\text{.}
\end{align*}
The upper images $\P$ of \eqref{vlp} and $\mathcal{M}$ of \eqref{vlp2} are depicted in Figure \ref{image_quad_ex}. Obviously both $(1,-1,1)^\transpose$ and $(-1,-1,-1)^\transpose$ solve the given problem with optimal value $-5$. Notice the solid recession cone of $\mathcal{M}$.

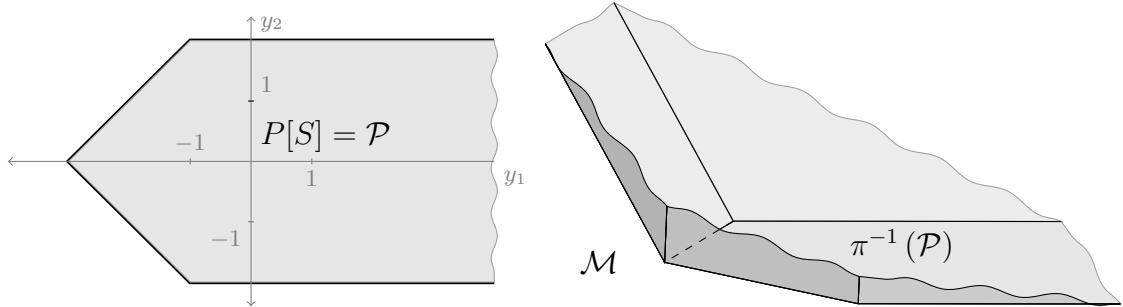
\begin{figure}[ht]
\begin{tikzpicture}[scale=0.8]
\draw [ultra thick] (4,-2)--(-1,-2)--(-3,0) -- (-1,2) -- (4,2);
\draw [gray!75!white,fill=gray!20!white] (-1,-2)--(-3,0)--(-3,0) -- (-1,2)--(4,2)decorate[decoration={snake,amplitude=.4mm,segment length=6mm}]{--(4,-2)}--cycle;
\draw [<-,gray] (-4,0) -- (4,0) ;
\draw [<->,gray] (0,-2.4) -- (0,2.4) ;
\draw [gray](-0.05,-1)-- (0.05,-1);
\node [gray,below left] at (0,-1) {\small $-1$};
\draw [](-0.05,1)-- (0.05,1);
\node [gray,above right] at (0,1) {\small $1$};
\draw [gray](-1,-0.05)-- (-1,0.05);
\node [gray,above] at (-1,0) {\small $-1$};
\draw [gray](1,-0.05)-- (1,0.05);
\node [below,gray] at (1,0) {\small $1$};
\node [above right,gray] at (0,2) {\small $y_2$};
\node [gray,below right] at (4,0) {\small $y_1$};
\node [above right] at (0,0) {\large $P[S]=\P$};
\end{tikzpicture}
\begin{tikzpicture}[scale=0.9,rotate around x=310, rotate around y=15, rotate around z=-320]

\coordinate (P3) at (-1,-2,3);
\coordinate (P2) at (-3 ,0,3);
\coordinate (P1) at (-1,2,-1);

\coordinate (D1) at (0,1,0);
\coordinate (D2) at (0,0,1);
\coordinate (D3) at (1,0,-1);

\coordinate (CA) at ($(P1)+5*(D3)$);
\coordinate (CB) at ($(P3)+4*(D3)$);
\coordinate (CC) at ($(P3)+0.8*(D2)$);
\coordinate (CD) at ($(P2)+4*(D1)$);
\coordinate (CE) at ($(P2)+1.65*(D2)$);
\coordinate (CF) at ($(P1)+4*(D1)$);

\def\amplitude{.5};
\def\snlen{10};

\draw[gray!75!white,thin,fill=gray!15!white] (P2)--(P1)--(CF)decorate[decoration={snake,amplitude=\amplitude mm,segment length=\snlen mm}]{--(CD)}--cycle;
\draw (CD)--(P2)--(P1)--(CF);
\draw[gray!75!white,thin,fill=gray!15!white] (P1)--(CA)decorate[decoration={snake,amplitude= \amplitude mm,segment length=\snlen mm}]{--(CF)}--cycle;
\draw (CF)--(P1)--(CA);
\draw[gray!75!white,thin,fill=gray!20!white] (P1)--(P2)--(P3)--(CB) decorate[decoration={snake,amplitude=\amplitude mm,segment length= \snlen mm}]{--(CA)}--cycle;
\draw(CA)--(P1)--(P2)--(P3)--(CB);
\draw[fill=gray!40!white, opacity=.5] (P3)--(CB)decorate[decoration={snake,amplitude=\amplitude mm,segment length=\snlen mm}]{--(CC)}--cycle;
\draw (CC)--(P3)--(CB);
\draw[fill=gray!50!white, opacity=.5] (P3)--(P2)--(CE)decorate[decoration={snake,amplitude=\amplitude mm,segment length=\snlen mm}]{--(CC)}--cycle;
\draw (CC)--(P3)--(P2)--(CE);
\draw[fill=gray!60!white, opacity=.5] (P2)--(CE)decorate[decoration={snake,amplitude=\amplitude mm,segment length=\snlen mm}]{--(CD)}--cycle;
\draw (CD)--(P2)--(CE);
\draw[dashed] (P1)--(P2);
\node at (1,0.8,-1.8) {\large $\pi^{-1}\left(\P\right)$};
\node at (-4,0,4) {\large $\mathcal{M}$};
\end{tikzpicture}

\caption{\label{image_quad_ex}%
  Image $P[S]$ and upper image $\P$ of \eqref{vlp} and upper image
  $\mathcal{M}$ of \eqref{vlp2} for Example~\ref{ex_nonsolid}. Notice
  the facet in $\mathcal{M}$ corresponding to $\P$.}
\end{figure}
\end{example}

\section{Numerical results} \label{sec_num}

The present section contains various numerical examples.  The package
{\em \bt} \cite{bt} for \octave\ / Matlab contains an implementation of the algorithms developed in this article.
The test problems in this section are solved using \bt\ with \octave\
on a computer with Intel\textregistered\ Core\texttrademark\ i7-6700HQ
CPU with 2.6 GHz.  For \bt\ we use the default tolerance for numerical
inaccuracies of $10^{-7}$.  We compare these results to the running
times achieved with \baron\ \cite{baron_2}, the general purpose solver
for mixed-integer nonlinear optimization problems.  The convergence
tolerance of \baron\ is also set to $10^{-7}$.

\subsection{Linear multiplicative programs}

Shao and Ehrgott \cite{linear_multiplicative} treat the problem class
of linear multiplicative programs. For
$c_i,l,u\in \R^n, b \in \R^m, d_i \in \R$ and $A \in \R^{m \times n}$,
they consider the problem
\begin{equation}\label{eq:lmp}\tag{LMP}
  \min \prod_{i=1}^q (c_i^\intercal x +d_i)  \quad\text{s.t.}\quad
  \left\{
    \begin{aligned}
      b &\leq Ax\text{,}\\
      l &\leq x \leq u\text{.}
    \end{aligned}
  \right.
\end{equation}
The parameters \(c_i,d_i\) of the objective function and the
constraint set in \eqref{eq:lmp} is chosen in such a way that
$c_i^\intercal x +d_i >0$ holds for all feasible points. In the following
example we generate random instances in the same way as Shao and Ehrgott
\cite{linear_multiplicative}.
\begin{example} \label{ex:lmp}
Let $A$ consist of equally distributed random real numbers out of the
interval $[0,10]$. The vectors $c_i$ and $b$ are generated in the same
way. The variable bounds are set to $l_j=0$ and $u_j=100$. For the
sake of simplicity, we set $d_i=0$.  We can now transform \eqref{eq:lmp} to an equivalent
problem of type \eqref{p}:
\[
  \min f(Px) \quad  \text{s.t.}\quad
  \left\{
    \begin{aligned}
      b &\leq Ax\text{,}\\
      l &\leq x \leq u\text{,}
    \end{aligned}
  \right.
\]
where the rows of $P$ consist of $c_i^\intercal$ from \eqref{eq:lmp}
and the objective \(f\) is defined as
\[
  f(y) =
  \begin{cases}
    \prod\limits_{i=1}^q y_i& \text{if \(y \in \R^q_+\),}\\
    -\infty& \text{else.}
  \end{cases}
\]
Then \(f\) is a quasi-concave function being $\R^q_+$-monotone on the
whole space $\R^q$.
\end{example}

The dual algorithm introduced in
\cite{linear_multiplicative} is similar to ours, but a different
vertex selection rule is used: While in line \ref{alg4_vsr} of
Algorithm \ref{alg_dual_modification} we determine a vertex $t^*$ of
$\O^*\setminus T^*$ such that $\varphi(t,t^*)$ is minimal, in
\cite[Algorithm 3.17, step (k1)]{linear_multiplicative} an arbitrary
vertex $t^*$ of $\O^*$ with $\varphi(t,t^*)<0$ is chosen. In order to
compare the different vertex selection rules, we also implement
Algorithm~\ref{alg_dual_modification} with the vertex selection rule
from \cite{linear_multiplicative}.  This modification is denoted by
Algorithm~\ref{alg_dual_modification}*.  In \autoref{table_mpp} we compare the running times of Algorithms
\ref{alg_primal_modification}, \ref{alg_dual_modification} and
\ref{alg_dual_modification}* to the times \baron\  needs to solve the
problem instances of Example~\ref{ex:lmp}.  For reference, we also
include the average running times for this problem class as reported
in \cite{linear_multiplicative}, where a personal computer with 2.5GHz
CPU and 4GB RAM is used for the computations.\par
\newcolumntype{C}{>{$}c<{$}}
\begin{table}[h]
\centering
\begin{tabular}{rCrrrrrr}
\toprule 
\multicolumn{6}{c}{}& \multicolumn{2}{c}{Shao/Ehrgott \cite{linear_multiplicative}}\\
$q$ & $(m,n)$ & \multicolumn{1}{c}{Alg.~\ref{alg_primal_modification}}& \multicolumn{1}{c}{{Alg.~\ref{alg_dual_modification}}} &\multicolumn{1}{c}{Alg.~\ref{alg_dual_modification}$^*$}& \multicolumn{1}{r}{BARON} & \multicolumn{1}{r}{primal} & \multicolumn{1}{r}{dual}\\
  \midrule
  \multirow{3}{*}{2}& (20,30)  &0.01 (7) &     0.01 (7) &     0.01 (7) & 0.18  &  0.10 & 0.11   \\
                    & (50,30)  &0.01 (7) &     0.01 (7) &     0.01 (7) & 0.28 & 0.14 & 0.11 \\
                    & (100,60) & 0.01 (8) &     0.01 (9) &     0.01 (9) & 1.09  & 0.20 & 0.15 \\
  \midrule
  \multirow{3}{*}{3} & (50,30) & 0.01 (18) &     0.01 (25) &     0.01 (32) & 0.50  & 0.34 & 0.29  \\
                     & (60,40) & 0.01 (18) &     0.01 (23) &     0.01 (32)  & 0.87 & 0.35 & 0.32 \\
                     & (100,60) & 0.01 (20) &     0.02 (24) &     0.02 (34) & 2.00  & 0.68 & 0.58   \\
  \midrule
  \multirow{2}{*}{4} & (60,40) & 0.02 (32) &     0.03 (51) &     0.04 (88) & -  & 2.09 &2.26  \\
                     & (100,60) &  0.03 (34) &     0.05 (56) &     0.08 (106) & -  & 7.98 & 7.94 \\
  \midrule
  5 & (100,60) & 0.07 (59) &     0.22 (101) &     0.69 (222) & -  & 24.17 & 29.38 \\ \hline
  \multirow{2}{*}{6}& (100,60) & 0.51 (87) &     5.17 (161) &     29.55 (410) & -  & 243.34 & 259.46  \\
                    & (150,80) & 0.82 (98) &     6.73 (183) &     64.83 (554)  &- &- & -\\
  \midrule
  \multirow{2}{*}{7} & (100,60) &19.89 (129) &     347.30 (263) &     -  &- &- & -\\
                    & (150,80) & 40.16 (164) &     384.23 (314) &     - & - &- & -\\
  \bottomrule

\end{tabular}
\caption{\label{table_mpp}%
  Average running time in seconds and number of iterations in
  parentheses for ten randomly generated instances of Example
  \ref{ex:lmp} of \autoref{alg_primal_modification} and
  \autoref{alg_dual_modification}.  \autoref{alg_dual_modification}*
  is a modified version of \autoref{alg_dual_modification} utilizing
  Shao and Ehrgott's vertex selection rule.  The two columns
  \textit{primal} and \textit{dual} contain the average running times
  of the respective algorithm taken from
  \cite{linear_multiplicative}.  The fourth column lists the
  average running times achieved by global optimization solver
  \baron. A `-' indicates that at least one test instance of the
  respective size was not solved within 600 seconds.
}
\end{table}  
The dual vertex \(t^*\) chosen in \autoref{alg_dual_modification},
line~\ref{alg4_vsr} may generate a half-space in the following
iteration step which contains the currently selected primal vertex
\(t\) (line~\ref{alg4_add_end}).  In this case the algorithm fails to
cut off the vertex \(t\) in line~\ref{alg4_cut}, and thus fails to
improve the current lower bound.  We call occurences of this case a
{\em failed cut}.  In \autoref{table_fail_cuts} we compare the number
of such failed cuts generated by
\autoref{alg_dual_modification} and the modification of this
algorithm with the vertex selection rule of
\cite{linear_multiplicative} (\autoref{alg_dual_modification}*).\par
\begin{table}[h]
\centering
\begin{tabular}{lrrrrr}
\toprule 
$q$ & 2 & 3& 4 & 5 & 6\\
\midrule
Alg. \ref{alg_dual_modification} & 0 & 1 & 5 & 13 & 21\\
Alg. \ref{alg_dual_modification}$^*$ & 1 & 8 & 50 & 133 & 294\\
\bottomrule
\end{tabular}
\caption{\label{table_fail_cuts}%
  Average number of failed cuts for Algorithm
  \ref{alg_dual_modification} and Algorithm
  \ref{alg_dual_modification}$^*$ for Example \ref{ex:lmp}. The data was
  generated by averaging the results of 10 different instances with
  100 constraints and 60 variables for each value of $q$.
}
\end{table}

\subsection{Concave quadratic programs}
Let $M\in \R^{n \times n}$ be a positive semi-definite
symmetric matrix with \(M = P^\transpose P\) for some matrix
\(P\in\R^{q\times n}\).  The problem
\begin{equation}\label{eq:cqp}\tag{CQP}
  \min_{x\in S} -x^\intercal Mx\text{,}
\end{equation}
where \(S\subseteq\R^n\) is a polytope, is a {\em concave quadratic
  optimization problem}.  Problem \eqref{eq:cqp} can be transformed to
\eqref{p} by using the concave objective function
\(f\colon\R^q\to \R\) defined by \(f(y)= -y^\intercal y\).  We obtain
\[
  -x^\intercal Mx = f(Px)\text{.}
\]
Monotonicity holds for the trivial cone $C= \{\,0\,\}$. Hence,
Algorithms \ref{alg_primal_modification} and
\ref{alg_dual_modification} can be applied to solve \eqref{eq:cqp}
using the techniques discussed in Section \ref{nonsolid_cones}.
\begin{example}[{\cite[Example 23]{dc_paper}}]\label{ex:cqp}
  For $q,n \in \mathbb{N}$, let $P \in \mathbb{R}^{q \times n}$ be defined as
  \begin{equation*}
    P_{ij}=\left\lfloor q \cdot \sin \left( \left( j-1 \right)\cdot q +i \right) \right\rfloor\text{,}
  \end{equation*}
  where
  $\lfloor x\rfloor \colonequals \max\{z \in \mathbb{Z}\mid z \leq
  x\}$.  Then $M \colonequals P^\transpose P$ is a positive
  semi-definite symmetric matrix.  We solve \eqref{p} with
  \(f(y)=y^\intercal y\), matrix \(P\) as defined above and feasible
  region
  $S=\left\{ x \in \mathbb{R}^n \,\middle\vert\, -e \leq x \leq e
  \right\}$.  We compare our results to the ones achieved by the
  non-convex problem solver \baron\ , see \cite{baron}, and to the
  results of the approach taken in \cite{dc_paper}. The numerical
  results are listed in Table \ref{table_quad}.
  \begin{table}[ht]
    \center
    \begin{tabular}{crrrrr}
      \toprule 
      $q$ &  $n$ & \text{Alg.~\ref{alg_primal_modification}}& \text{Alg.~\ref{alg_dual_modification}} & \text{\baron} & \text{DC(\cite{dc_paper})}\\
      \midrule
      \multirow{3}{*}{2} & 200 & 0.01 & 0.01 & 0.33 & 0.08\\
          & 1000 & 0.04 & 0.03 & 6.08 & 0.49 \\
          & 5000 & 0.85 & 0.67 & 567.96 & 10.48 \\ 
      \midrule
      \multirow{3}{*}{3} & 200 & 0.03 & 0.02 & 0.56 & 0.13\\
          & 1000 & 0.14 & 0.06 & 21.16 & 1.16 \\
          & 5000 & 2.98 & 1.53 & - & 25.40 \\
      \midrule
      \multirow{3}{*}{4} & 200 & 0.05 & 0.04 & 1.88 & -\\
          & 1000 & 0.42 & 0.16 & 43.84 & -\\
          & 5000 & 8.40 & 3.15 & - & -\\
      \midrule
      \multirow{5}{*}{5} & 10 & 0.02 & 0.05 & 0.22 & 0.06\\
          & 50 & 0.38 & 0.24 & 1.43 & - \\
          & 200 & 0.05 & 0.22 & 6.76 & - \\
          & 1000 & 0.47 & 0.50 & 109.06 & -\\
          & 5000 & 10.41 & 3.99 & - & - \\
      \midrule
      \multirow{5}{*}{6} & 10 & 0.83 & 0.55 & 0.24 & 0.09\\
          & 50 & 18.32 & 3.92 & 0.46 & - \\
          & 200 & 1.06 & 16.89 & 68.19 & -\\
          & 1000 & 2.58 & 20.50 & - & - \\
          & 5000 & 22.4 & 28.64 & - & -\\
      \bottomrule
    \end{tabular}
    \caption{\label{table_quad}%
      Numerical results for the concave quadratic program of Example
      \ref{ex:cqp}. Running times are given in seconds. Again, no number
      given in the column corresponding to \baron\  indicates the exceeding
      of 600 seconds given. The last column contains results of
      \cite{dc_paper} obtained by a DC-programming reformulation and using
      an unmodified MOLP solver.}
  \end{table}  

\end{example}

\subsection{DC-programs}
Recall the problem class of DC-programs with one polyhedral component introduced in Examples \ref{ex01} and \ref{ex02} in the introduction, which were shown to be special cases of \eqref{p}.

\begin{example} \label{dc_example}
We want to solve the following problem from \cite{Ferrer2015} and discussed in \cite{dc_paper}:
\begin{equation}\label{eq:exdc}
  \min_{x \in S} g(x)-h(x) \text{,}
\end{equation}
where
\begin{align*}
  &g(x)=\lvert x_1 -1 \rvert +200 \sum_{i=2}^q \max \left\{ 0,  \lvert
    x_{i-1} \rvert-x_i \right\}
  &\text{and}&
  &h(x)=100 \sum_{i=2}^q  \left( \lvert x_{i-1} \rvert -x_i \right)\text{.}
\end{align*}
The feasible region is
$S=\left\{ x \in \mathbb{R}^q \,\middle\vert\, -10 \cdot e \leq x \leq
  10 \cdot e \right\}$.  Both of the given functions are
polyhedral. Hence, by following the procedure in the introducing section
(Examples \ref{ex01} and \ref{ex02}) we obtain the two equivalent
problems
\begin{align*}
\min r -h(x) \quad &\text{s.t.\ \((x,r) \in \epi \hat{g}\)}
\intertext{and}
\min r^* -\hat{g}^*(x^*) \quad &\text{s.t.\ \((x^*,r^*) \in \epi h^*\),}
\end{align*}
where we set
\[
  \hat{g}(x) \colonequals
  \begin{cases}
    g(x)& \text{if \(x \in S\),}\\
    +\infty& \text{else.}
  \end{cases}
\]
We can now solve our initial problem by solving one of the two
problems above.  They both have a polyhedral feasible region and
concave objective functions.  The objectives of both problems are
monotone with respect to the cone \(C =
\left\lbrace(0,\ldots,0,t)^\transpose\in\R^{q +
    1}\,\middle\vert\,t\geq 0\right\rbrace\).  As \(\inter C = \emptyset\),
this problem is solved by using the extension discussed in Section
\ref{nonsolid_cones}.  The optimal value of \eqref{eq:exdc} is \(0\),
and a solution is given by \(e\in\R^q\).  In \autoref{table_dc} we list
numerical results for \autoref{alg_primal_modification} compared to
the ones obtained in \cite{dc_paper} and \cite{Ferrer2015}. \baron\ 
solves any instance of this problem in $0.01$ seconds.  This is
probably due to the simple structure of the solution.
\end{example}
\begin{table}[ht]
\center
\begin{tabular}{rrrrrrr}
\toprule 
p & \text{DCECAM  \cite{Ferrer2015}} & \text{DCPA \cite{Ferrer2015}} & \text{DC \cite{dc_paper}} & \text{DC$^*$ \cite{dc_paper}} & \text{QCP} & \text{QCP}$^*$\\
\midrule
2 &0.21&0.22&     0.05 & 0.05        &  0.01 & 0.01\\
3 &3.57&4.63&     0.06 & 0.06        &  0.01 & 0.02\\
4 &2.47&0.78&     0.17 & 0.08        &  0.01 & 0.03\\
5 &345.12&502.29& 3.68 & 0.09        & 0.30  & 0.06\\
6 &-&-&           375.03 & 0.13      & 12.34  & 0.11\\
7 &-&-&          28003.06      & 0.21       & 350.88 & 0.23\\
8 &-&-&           -      & 0.80     &  -   & 0.47\\
9 &-&-&           -      & 7.95     &-      & 0.99\\
\bottomrule

\end{tabular}
\caption{\label{table_dc}%
  Running time in seconds for Example \ref{dc_example}. The first two
  columns are results of \cite{Ferrer2005} obtained by the DC extended
  cutting angle method (DCECAM) and the DC prismatic algorithm
  (DCPA). It should be pointed out that these two methods do not
  require one of the two objective functions to be polyhedral. Thus,
  they are capable of solving more general problems than we do in this
  article. The next two columns, DC and DC$^*$, are results of
  \cite{dc_paper} obtained by using an unmodified MOLP solver. The
  last two columns, QCP and QCP$^*$, are results obtained by the
  extension of Algorithm 2 using the cone
  $\R_+ \cdot \left( 0,\ldots,0,1 \right)^\intercal$. DC and QCP are
  based on the primal approach in Example \ref{ex01}. DC$^*$ and
  QCP$^*$ are based on the dual approach outlined in Example
  \ref{ex02}.}
\end{table}

\subsection{Minimizing a convex function over the boundary of a polytope} \label{subsec:bd}

Example \ref{ex03} in the introduction motivates the class \eqref{p}
by the problem to minimize a Lipschitz continuous convex function $g$
over the boundary of a polytope $Q$.  To this end, the optimization
problem \eqref{p_ex03_a} is reformulated as the DC optimization
problem \eqref{p_ex03_b}, which depends on a sufficiently large
constant $c>0$. Let $L>0$ be the Lipschitz constant of $g$ as a
function defined on $Q$, that is,
\begin{equation}\label{eq_lip}
  \forall x,y \in Q\colon g(x)-g(y) \leq L \lVert x-y\rVert\text{,}
\end{equation}
where $\lVert\cdot\rVert$ denotes the Euclidean norm. The next
statement provides a proper choice of the parameter $c$ in dependence
of $L$.
\begin{proposition} \label{prop_cc}
  Let $Q=\left\{ x \in \R^q \,\middle\vert\,
    \exists u \in \R^k\colon A x + B u \geq b\right\}$ be a polytope with
  $0 \in \inter Q$ and let $g\colon \R^q \to \R\cup\{\,+\infty\,\}$ be
  a convex function such that \eqref{eq_lip} holds. Let \(R\in\R\) with
  \begin{equation*}
    R \geq \max\left\{\lVert x\rVert \,\middle\vert\, x \in Q\right\}\text{.}
  \end{equation*} 	
  When the parameter $c$ in Problem \eqref{p_ex03_b} is chosen such
  that $c > L R$, then \eqref{p_ex03_a} and \eqref{p_ex03_b} have the
  same set of optimal solutions and the same optimal value.
\end{proposition}
\begin{proof}
  First note that the objective functions of \eqref{p_ex03_a} and
  \eqref{p_ex03_b} coincide on the boundary of $Q$. Hence, it suffices
  to show that every optimal solution of \eqref{p_ex03_b} belongs to
  the boundary of $Q$. Assume that an optimal solution $x^*$ of
  \eqref{p_ex03_b} belongs to the interior of $Q$. We start with the
  case where $x^* \neq 0$. There exists $\mu > 1$ such that
  $\mu x^* \in Q$ and we have
\begin{equation*}
	g(\mu x^*) - (h_c(\mu x^*) - c) \geq  g(x^*) - (h_c(x^*)-c)\text{.}	
\end{equation*}	
Thus
\begin{equation*}
	L (\mu-1) \lVert x^* \rVert \geq g(\mu x^*) - g(x^*) \geq h_c(\mu x^*) - h_c(x^*) \geq \frac{c}{R} (\mu-1) \lVert x^* \rVert\text{,}
\end{equation*}
where the latter inequality follows from the fact that the epigraph of $h_c$ is the cone generated by the set $Q \times \{c\}$. Hence $L R \geq c$, which contradicts the assumption $LR < c$. The case $x^*=0$ can be shown likewise by replacing $\mu x^*$ by some arbitrary $x \in Q \setminus \{0\}$.	
\end{proof}

\begin{example}\label{ex28}
  Let $q,m$ be positive integers with $q\leq m$ and let
  $P \in \mathbb{R}^{q \times m}$ be the matrix described in Example
  \ref{ex:cqp}. Furthermore let
  $S=\left\{ u \in \mathbb{R}^m \, | \, -e \leq u \leq e \right\}$. We
  intend to solve the problem
  \begin{equation}\label{bd_problem}
    \min_{x \in \bd P[S]} x^\intercal x\text{.}
  \end{equation}
  Note that the polyhedron $Q\colonequals P[S]$ is given by a P-representation
  \begin{equation*}
    Q=\left\{ x \in \mathbb{R}^p \,\vert\, \exists u \in \R^m\colon u \in S , x=Pu  \right\}\text{.}
  \end{equation*}
  Following the procedure described in Example \ref{ex03} we obtain a
  DC optimization problem as considered in Example \ref{ex02}. We set
  \begin{equation*}
    g(x)=
    \begin{cases}
      x^\intercal x& \text{if \(x \in Q\),} \\
      \infty& \text{otherwise.}
    \end{cases}
  \end{equation*}
  Thus, the conjugate in \eqref{eq_tsd} is obtained by solving the
  quadratic convex program
  \begin{equation}\label{eq_qp}
    -g^*(y) = \min_{x \in Q} \{\,x^\intercal x - y^\intercal x\,\}\text{.}
  \end{equation}
  The definition of the polyhedral convex function \(h\) in
  Example~\ref{ex03} requires the parameter \(c\) being sufficiently
  large.  Let $r$ be the vector of row sums of absolute values in $P$.
  We choose \(c\), according to Proposition~\ref{prop_cc} with
  \(R = \lVert r\rVert\) and Lipschitz-constant
  \(L = 2\lVert r\rVert\) of \(g\), as
  $c\colonequals 2 \lVert r \lVert^2+1$.  A representation of $h^*$,
  as needed in \eqref{eq_tsd}, is obtained as described in
  \cite[Proposition 6]{dc_paper}.\par
  Numerical results of \autoref{alg_primal_modification} applied to
  this problem are listed in Table \ref{table_bd}.  Problem
  \eqref{bd_problem} cannot be solved by \baron\  in the way described
  above, as \baron\ requires explicitly expressed algebraic functions, see
  \cite{baron}.
\end{example}

\begin{table}[ht]
\center
\begin{tabular}{rrrrrr}
\toprule 
 & \multicolumn{5}{C}{q} \\
 \cmidrule[0.2pt](lr){2-6}
m & 1 & 2& 3 & 4 & 5 \\
\midrule
10 & 0.10   & 0.48 & 0.67 & 2.21 & 44.27\\
20 & 0.15   & 0.84 & 1.08 & 5.84 & 26.71\\
50 & 0.72   & 1.87 & 4.64 & 41.19 & 159.97\\
100 & 3.84  & 11.48 & 20.17 & 152.62 & 2316.05 \\
200 & 25.35 & 77.98 & 93.80 & 1504.82 & 33089.42\\
\bottomrule
\end{tabular}
\caption{\label{table_bd}%
  Running time in seconds for Example \ref{ex28} using Algorithm
  \ref{alg_primal_modification} and Octaves's sqp solver for solving
  \eqref{eq_qp}.}
\end{table}

\section{Conclusion}
The contribution of this article can be summarized as follows:

We generalize the approach of Mittal and Schulz \cite{mittal_schulz} with respect to
the following three aspects: First, the objective function is not supposed to have a
certain scaling property at the price of loosing polynomial running time.
Secondly, our approach is based on Benson-type algorithms for MOLPs instead
of using grid-based scalarization parameters. Thirdly, we allow polyhedral
ordering cones $C$ which are more general than $\mathbb{R}^q_+$ in
order to weaken the monotonicity assumption to the objective function.
In particular, in Section \ref{nonsolid_cones} we even allow
the cone $C=\{0\}$, which means that no monotonicity assumption
is required. We present a technique that allows to treat the case of
$\inter C = \emptyset$ even though the VLP solver requires
an ordering cone $C$ with nonempty interior.

The results of Shao and Ehrgott \cite{linear_multiplicative} for
multiplicative linear programs (compare Example~\ref{ex04}) are generalized
to the class \eqref{p}. Moreover, we suggest an improvement of
the dual algorithm introduced in \cite{linear_multiplicative},
which consists of a vertex selection rule based on the strongest
violation of an optimality condition.

The results of \cite{dc_paper}, where a MOLP solver without any
modification was used to solve the problem classes of
Examples~\ref{ex01} and \ref{ex02} are generalized and improved,
since the approach we introduced requires less iteration
steps, in general. Numerical examples show that our approach via (a modified)
VLP solver is competitive with the global optimization
software \baron\  \cite{baron}.


\end{document}